\documentclass[12pt,a4paper]{article}
\usepackage[utf8]{inputenc}
\usepackage[T1]{fontenc}
\usepackage{amsmath,amssymb,amsthm,mathtools}
\usepackage{graphicx}
\usepackage[a4paper,hmargin=3cm,vmargin=3cm]{geometry}
\usepackage[expansion=false]{microtype}
\usepackage[colorlinks=true,linkcolor=blue,citecolor=blue,urlcolor=blue]{hyperref}
\graphicspath{{figures/}{./}}

\newcommand{\figplaceholder}[2]{%
  \fbox{\begin{minipage}[c][0.30\textheight][c]{#1}%
    \centering\ttfamily\small [missing figure file: #2]\\[6pt]%
    \normalfont\itshape\footnotesize
    Upload \texttt{#2} together with the manuscript source to render this figure.%
  \end{minipage}}}
\newcommand{\robustfig}[2][\linewidth]{%
  \IfFileExists{#2}{\includegraphics[width=#1]{#2}}%
  {\IfFileExists{figures/#2}{\includegraphics[width=#1]{figures/#2}}%
  {\figplaceholder{#1}{#2}}}}
\allowdisplaybreaks
\providecommand{\tightlist}{\setlength{\itemsep}{0pt}\setlength{\parskip}{0pt}}

\numberwithin{equation}{section}
\theoremstyle{plain}
\newtheorem{theorem}{Theorem}[section]
\newtheorem{proposition}[theorem]{Proposition}
\newtheorem{lemma}[theorem]{Lemma}
\newtheorem{corollary}[theorem]{Corollary}
\theoremstyle{definition}
\newtheorem{definition}[theorem]{Definition}
\newtheorem{assumption}[theorem]{Assumption}
\theoremstyle{remark}
\newtheorem*{remark}{Remark}

\title{Spectral-Capacitary Bounds for Homological Recovery from Reflected Brownian Trajectories}
\author{Tristan Guillaume\thanks{CY Cergy Paris Universit\'e.}}
\date{}

\begin{document}
\maketitle

\begin{abstract}
We study how long a single reflecting Brownian trajectory must be
observed before the topology of its unknown domain can be recovered from
the path alone. Formulating this as a minimax estimation problem for a
compact reflecting domain observed through the reflected Wiener sausage,
we identify the optimal observation time: up to constants,
\(T^{*} \asymp s_{0} + \kappa_{d}(\varepsilon)^{-1}\log m\). This
combines a spectral-access term \(s_{0} \asymp \lambda_{1}(M)^{-1}\) set
by the Neumann spectral gap, the inverse Brownian capacity
\(\kappa_{d}(\varepsilon)^{-1}\) of a feature at scale \(\varepsilon\)
--- capacity, not volume, governing detection --- and a logarithmic
factor \(\log m\) in the feature complexity. Matching minimax lower
bounds, from dumbbell and hidden-feature constructions, show all three
terms necessary. A spectral hitting bound through killed eigenvalues
underlies these results and yields two regimes: a full intrinsic
\(\varepsilon\)-net forces an \(\varepsilon\)-dense trace and a two-sided
homology isomorphism in every degree (full reconstruction), whereas
witness regions around one robust feature force only a surjection
(detection) at a faster rate. The analytic core is a uniform small-hole
eigenvalue theorem comparing the principal killed eigenvalue to Brownian
capacity uniformly over all centres, including those meeting the
reflecting boundary.
\end{abstract}



\section{Introduction}
\subsection{Topological inference from a single reflected trajectory}
Suppose that a particle moves randomly inside an unknown compact region,
reflecting at its boundary, and that an observer records the particle's
position over time but never sees the boundary, the obstacles, or the
holes that constrain the motion. This paper studies a statistical
inference problem: from the single observed trajectory, can the topology
of the hidden domain be recovered, and how long must the particle be
watched before recovery succeeds with high probability? We cast it as a
minimax problem --- the parameter is the unknown domain, the datum is
one reflected Brownian trajectory, and the risk is the probability of
misidentifying the homology --- and we determine the optimal observation
time.
The answer is governed by three quantities of distinct geometric origin.
Writing \(s_{0}\) for a heat-kernel burn-in time,
\(\kappa_{d}(\varepsilon)\) for the Brownian capacity scale of an
\(\varepsilon\)-feature, and \(m\) for a measure of topological feature
complexity, the minimax-optimal observation time is, up to constants,
\[T^{*}\  \asymp \ s_{0} + \kappa_{d}(\varepsilon)^{- 1}\log m\ ,\]
the sum of a global spectral-access term \(s_{0}\), controlled by the
Neumann spectral gap of the domain and comparable in regular settings to
its inverse \(\lambda_{1}(M)^{- 1}\); a local capacitary-detection term
\(\kappa_{d}(\varepsilon)^{- 1}\) --- the time for the diffusion to hit
one \(\varepsilon\)-scale feature, set by Brownian capacity rather than
by volume; and a feature-complexity logarithm \(\log m\), equal to
\(\log N_{\varepsilon}(M)\) (the intrinsic covering number) for full
reconstruction and to the logarithm of the number of robust features for
their detection (Section 1.3). Matching minimax lower bounds show that
none of the three quantities can be removed. The analytic engine
throughout is the potential theory of reflected Brownian motion ---
Neumann and killed eigenvalues, small-target capacity, and heat-kernel
domination --- which serves here as the tool, not the object, of study.
Let \(M \subset \mathbb{R}^{d}\) be the unknown compact domain and let
\(X = \left( X_{t} \right)_{t \geq 0}\) be reflected Brownian motion in
\(M\). The observed trace up to time \(T\) is
\[\Gamma_{T} = \{ X_{t}:0 \leq t \leq T\},\]
and for a reconstruction radius \(r > 0\) the observed object is the
reflected Wiener sausage
\[\Gamma_{T}^{r} = \{ x \in \mathbb{R}^{d}:\operatorname{dist}\left( x,\Gamma_{T} \right) \leq r\}.\]
This offset is taken in the ambient space \(\mathbb{R}^{d}\), not
intersected with \(M\), in the sense of the Euclidean offset introduced
in Section 2.
The central problem is to determine when, for a prescribed homological
degree \(q\),
\[H_{q}\left( \Gamma_{T}^{r} \right) \cong H_{q}(M)\]
with high probability. The formulation has two logically distinct
components. The deterministic component is classical: samples that are
dense below the reach scale of \(M\) recover homology through offsets,
by the theory of sets of positive reach and geometric inference
originating in Federer {[}1{]} and developed quantitatively by Niyogi,
Smale and Weinberger {[}2{]} and Chazal, Cohen-Steiner and Lieutier
{[}3{]}. The stochastic component is the subject of this paper: the
sample is not an independent point cloud but the continuous trace of a
single dependent diffusion, and the task is to estimate the time at
which that trace becomes topologically informative.
\begin{figure}[htbp]
\centering
\robustfig[0.98\linewidth]{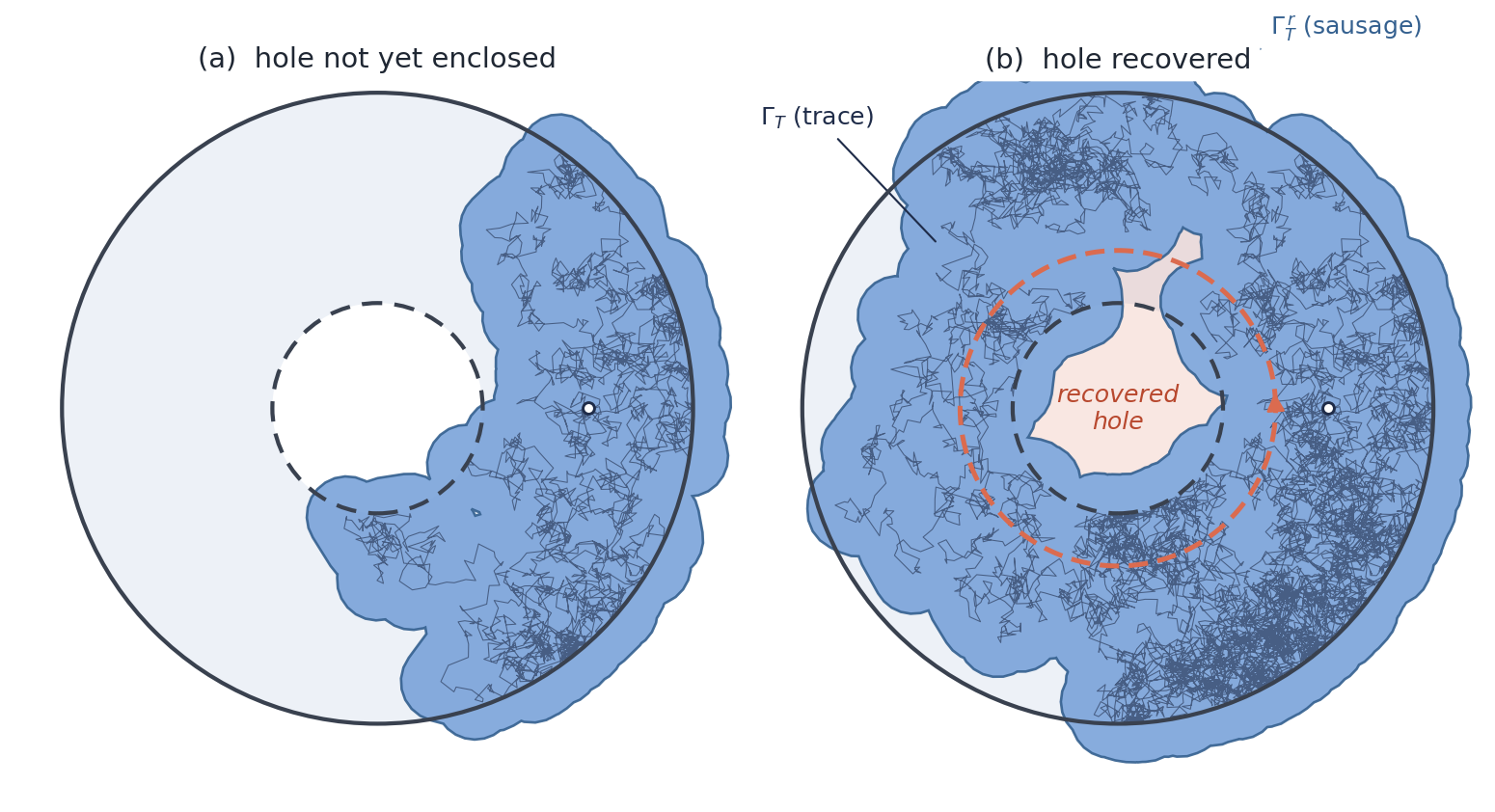}
\caption{Recovering a hole from a single reflected trajectory. A reflected Brownian motion in an annular domain \(M\) (dark path, the trace \(\Gamma_{T}\)) is thickened to its \(r\)-offset, the reflected Wiener sausage \(\Gamma_{T}^{\, r}\) (shaded). \textbf{(a)} After a short observation the sausage has not yet wound around the central hole and is contractible. \textbf{(b)} Once the trace encircles the hole, the sausage carries a loop around it (dashed), so \(H_{1}\left( \Gamma_{T}^{\, r} \right) \cong H_{1}(M)\) and the hole is recovered. What sets the time this takes is the Brownian capacity of the obstacle, not its area.}
\label{fig:1}
\end{figure}
\subsection{Why capacity, not volume, controls detection}
A crude approach would divide time into approximately independent mixing
blocks and ask whether the process lies inside a given
\(\varepsilon\)-ball at the end of one of these blocks. Since an
\(\varepsilon\)-ball has volume of order \(\varepsilon^{d}\), this
suggests a local detection cost of order \(\varepsilon^{- d}\). That is
not the correct scale. A Brownian path can hit a small set during a time
interval without being inside it at the endpoint of the interval;
hitting is governed by potential theory, not by stationary volume. The
relevant quantity is Brownian capacity, whose dimensional scale for a
ball of radius \(\varepsilon\) is
\[\kappa_{d}(\varepsilon) = \left\{ \begin{matrix}
\left( \log\varepsilon^{- 1} \right)^{- 1}, & d = 2, \\
\varepsilon^{d - 2}, & d \geq 3.
\end{matrix} \right.\ \]
The inverse scale \(\kappa_{d}(\varepsilon)^{- 1}\) --- logarithmic in
dimension \(2\), of order \(\varepsilon^{2 - d}\) in dimensions
\(d \geq 3\) --- is the time on which a mixed reflected Brownian motion
hits one small target. In dimension \(2\), combined with a
coupon-collector logarithm over a net of size \(N_{\varepsilon}(M)\),
this produces the logarithmic-square order familiar from two-dimensional
cover-time theory {[}4{]}. This capacitary viewpoint is classical in
potential theory and the study of Brownian obstacles {[}5, 6{]}; the
present paper uses it for an inverse purpose, the recovery of hidden
topology from a single reflected trajectory, and the matching lower
bounds show that capacity is not an artifact of the method.
\subsection{Main results}
\textbf{1. Minimax-optimal rate.} We determine, up to constants, the
optimal observation time for recovering the homology of an unknown
reflecting domain from a single reflected Brownian trajectory, with
matching upper and lower bounds exhibiting the structure
\(T^{*} \asymp s_{0} + \kappa_{d}(\varepsilon)^{- 1}\log m\) (Theorem
6.3 and Section 7).

\textbf{2. A feature-adaptive certificate framework.} We introduce
\emph{homological certificates} (Section 3), a deterministic device
through which a single spectral hitting estimate (Theorem 4.4) governs
both two-sided reconstruction from a covering net and one-sided
detection of individual robust features from a bounded witness family
(Definition 3.5), the covering logarithm \(\log N_{\varepsilon}(M)\)
reducing to \(\log m\) for detection (Theorem 6.1).

\textbf{3. A uniform small-hole eigenvalue theorem.} We prove that the
principal eigenvalue of the reflected generator killed on a small
intrinsic ball satisfies
\(\lambda\left( B_{M}(x,\varepsilon) \right) \asymp \kappa_{d}(\varepsilon)\)
uniformly over all centres, including targets meeting the reflecting
boundary (Theorem 5.3) --- the analytic core of the capacitary term ---
together with a sharp first-order asymptotic under added boundary
regularity (Appendix D).

\textbf{4. Matching minimax lower bounds.} We construct families of
admissible domains showing each quantity is unavoidable: a bottleneck
forces the spectral-access term \(\lambda_{1}(M)^{- 1}\), one hidden
feature forces the inverse-capacity term
\(\kappa_{d}(\varepsilon)^{- 1}\), independent hidden features force the
coupon-collector logarithm, and a combined dumbbell construction shows
the obstructions can be simultaneously necessary (Propositions 7.3, 7.4,
7.6 and Theorem 7.8).

\textbf{5. Consequences for persistence and discrete observation.} We
deduce approximation of persistence diagrams and recovery from
discretely sampled trajectories at the same observation time (Section
8).

We now describe these results in more detail.
The deterministic backbone is the notion of a \emph{homological
certificate} (Definition 3.3): a finite family
\(A = \{ A_{1},\ldots,A_{m}\}\) of target regions whose visitation by a
path forces its \(r\)-offset to carry the correct degree-\(q\) homology,
in the two-sided sense
\(H_{q}\left( \operatorname{Tr}(\gamma)^{r} \right) \cong H_{q}(M)\). A full intrinsic
\(\varepsilon\)-net is such a certificate (Proposition 3.4(ii)), and the
only one we exhibit in general: a path hitting every ball of a net is
forced to be \(\varepsilon\)-dense, and density yields the two-sided
isomorphism. A strictly smaller family --- for instance a bounded ring
of witness regions encircling one obstacle in a planar domain --- need
not be a certificate in this sense; what it forces is a weaker,
one-sided guarantee we call \emph{detection}, namely that the prescribed
classes are realized in the offset, with spurious classes not excluded
(Section 3.2). Reconstruction and detection carry the same probabilistic
cost, the time to visit the family, and differ only in their
deterministic reading.
The probabilistic engine is a finite-target spectral hitting estimate
(Theorem 4.4). Writing \(\lambda_{A}\) for the principal eigenvalue of
the reflected generator killed on a target \(A\), and
\[\Phi_{\mathcal{A}}(u) = \sum_{A\mathcal{\in A}}^{}e^{- \lambda_{A}u}\]
for the spectral profile of the family, the probability that some target
remains unvisited at time \(T\) is at most
\(a_{hk}\,\Phi_{\mathcal{A}}\left( T - s_{0} \right)\), where \(s_{0}\)
and \(a_{hk}\) are the heat-kernel domination constants of \(M\). The
main upper bound (Theorem 6.1) is the conjunction of these two
ingredients: if \(\mathcal{A}\) is a degree-\(q\), radius-\(r\)
certificate, then for every starting point \(z \in M\),
\[\mathbb{P}_{z}\left( H_{q}\left( \Gamma_{T}^{r} \right) \cong H_{q}(M) \right) \geq 1 - a_{hk}\,\Phi_{\mathcal{A}}\left( T - s_{0} \right).\]
The explicit rates follow from the central analytic result of the paper,
a uniform small-hole eigenvalue theorem (Theorem 5.3): there exist
constants such that, for all sufficiently small \(\varepsilon\),
\[\lambda_{B_{M}(x,\varepsilon)} \asymp \kappa_{d}(\varepsilon)\quad\text{uniformly in }x \in M,\]
including targets meeting the reflecting boundary. For a family of \(m\)
ball-like witness regions at scale \(\varepsilon\) that is a
degree-\(q\), radius-\(r\) certificate in the sense of Definition 3.3
--- for instance a full \(\varepsilon\)-net (Proposition 3.4(ii)) ---
recovery in degree \(q\) holds with probability at least \(1 - \delta\)
once
\[T \geq C\left( s_{0} + \kappa_{d}(\varepsilon)^{- 1}\left( \log m + \log\frac{1}{\delta} \right) \right).\]
A family smaller than a full net is, in general, not known to be a
certificate in this two-sided sense; what such a family forces is the
one-sided guarantee of Section 3.2 --- that the prescribed degree-\(q\)
classes are realized in the offset, i.e.~surjectivity of the inclusion
\(H_{q}\left( \Gamma_{T}^{r} \right) \rightarrow H_{q}(M)\). We call
this detection, and the same time bound, now with \(\log m\) in place of
\(\log N_{\varepsilon}(M)\), governs it. Whether a family strictly
smaller than a full net can force the two-sided isomorphism in an
intermediate degree \(0 < q < d\) is left open (Remark after Proposition
3.4).
Taking the certificate to be a full intrinsic \(\varepsilon\)-net, of
size \(N_{\varepsilon}(M) \asymp \varepsilon^{- d}\), yields the
universal theorem (Theorem 6.3): \emph{all} homology groups are
recovered, throughout a stable window of reconstruction radii
\(C_{low}\,\varepsilon \leq r \leq C_{up}\, \operatorname{reach}(M)\), once
\[T \geq C\left( s_{0} + \kappa_{d}(\varepsilon)^{- 1}\left( \log N_{\varepsilon}(M) + \log\frac{1}{\delta} \right) \right).\]
In dimension \(2\) the universal scale is
\(s_{0} + \left( \log\varepsilon^{- 1} \right)^{2}\); in dimensions
\(d \geq 3\) it is \(s_{0} + \varepsilon^{2 - d}\log\varepsilon^{- 1}\).
The logarithmic factor is thus governed by topological feature
complexity, not by covering as such: for the detection of finitely many
robust features it reduces to the logarithm of the number of witness
regions, while full reconstruction continues to require the covering
logarithm, as the matching lower bounds confirm.
Finally, Section 7 proves minimax lower bounds showing that the three
terms of the recovery scale are genuine. Bottleneck (dumbbell) domains
force a spectral access term of order \(\lambda_{1}(M)^{- 1}\), the
inverse Neumann spectral gap (Proposition 7.3). A single feature hidden
in an \(\varepsilon\)-ball forces the inverse-capacity term
\(\kappa_{d}(\varepsilon)^{- 1}\) (Proposition 7.4). Families of \(m\)
independent hidden features force the coupon-collector term
\(\kappa_{d}(\varepsilon)^{- 1}\log m\) (Proposition 7.6), and hence a
full-cover lower bound matching the universal scale (Corollary 7.7). A
combined dumbbell construction with hidden small-scale features in the
far chamber shows that spectral access and capacitary coupon-collecting
can be simultaneously necessary (Theorem 7.8). The resulting picture is
that homological recovery from a single reflected Brownian trajectory is
governed by three independent quantities: global spectral access, local
capacitary detection, and topological feature complexity.
\subsection{Related work}
The deterministic engine is offset reconstruction from dense samples:
Federer's theory of positive reach {[}1{]}, the homotopy-recovery
theorem of Niyogi, Smale and Weinberger {[}2{]}, and its extension to
compact sets through critical values of distance functions by Chazal,
Cohen-Steiner and Lieutier {[}3{]}. Our use of this theory is standard;
the novelty lies in the stochastic question of when one reflected
trajectory enters the reconstruction window. The persistence
consequences in Section 8 rest on the stability theory of persistence
diagrams {[}7, 8, 9, 10{]}.
The statistical aspect is closest to minimax homology inference from
independent samples, studied by Balakrishnan, Rinaldo, Sheehy, Singh and
Wasserman {[}11{]}. The decisive difference is the sampling mechanism:
our data are the continuous trace of a single reflected diffusion
constrained by the unknown domain, and this dependence changes the
governing quantities to the Neumann spectral gap, killed eigenvalues,
and Brownian capacity. The analytic inputs come from the theory of
reflected Brownian motion and symmetric Dirichlet forms: existence and
potential theory in nonsmooth domains go back to Bass and Hsu {[}12{]},
the Dirichlet-form and capacity machinery is developed in Fukushima,
Oshima and Takeda {[}13{]}, and heat-kernel domination after burn-in
follows from Neumann heat-kernel estimates of the type studied by Gyrya
and Saloff-Coste {[}14{]}.
That small-target hitting is governed by capacity rather than volume is
classical: it appears in Spitzer's work on capacity and Brownian motion
{[}5{]}, in the Wiener-sausage asymptotics of Donsker and Varadhan
{[}15{]} and their refinements {[}16{]}, in the capacity of the range of
random walk {[}17{]}, and in Sznitman's enlargement-of-obstacles
framework {[}6{]}. The small-hole eigenvalue theorem belongs to the
spectral perturbation tradition of Rauch and Taylor {[}18{]} and Ozawa
{[}19{]}. Our contribution at this point is not the existence of
capacitary spectral asymptotics in isolation, but the \emph{uniform
reflected-domain} estimate needed to apply finite-target hitting bounds
over all centers, including targets meeting the Neumann boundary. The
universal full-cover regime is directly comparable to the Brownian
cover-time theory of Dembo, Peres, Rosen and Zeitouni {[}4{]}, whose
two-dimensional logarithmic-square order our universal theorem recovers
at the level needed for homological reconstruction; we do not claim to
improve sharp cover-time theory. Finally, the problem is complementary
to the direct study of the topology of random sets such as random
geometric complexes {[}20, 21{]}: there the random set creates the
topology of interest, whereas here the sausage is an observational
instrument and the random holes left by incomplete exploration are
nuisance features, not the target of inference.
\subsection{Organization}
Section 2 fixes the setting: admissible reflecting domains, positive
reach, intrinsic and extrinsic density, reflected Brownian motion,
heat-kernel domination, and the capacity scale. Section 3 develops the
deterministic bridge: offset reconstruction below the reach and
homological certificates. Section 4 proves the finite-target spectral
hitting estimate. Section 5 proves the uniform small-hole eigenvalue
theorem; the local condenser and boundary-flattening estimates it relies
on are collected in Appendix A. Section 6 proves the main upper bounds,
with universal recovery derived as the extremal certificate. Section 7
proves the lower bounds; the pairwise decorrelation lemma and the
second-moment supplement for the many-feature construction are proved in
Appendix B; the killed-eigenvalue and eigenfunction estimates underlying
the bottleneck and dumbbell lower bounds are collected in Appendix C.
Section 8 records persistence and discrete-observation consequences.
Under an additional uniform \(C^{1,1}\) regularity assumption on
\(\partial M\), Appendix D supplements this with a first-order analytic
refinement, identifying the sharp leading constant of the small-hole
eigenvalue together with its interior and boundary-layer profile; the
main recovery theorems themselves rely only on the uniform comparability
of Theorem 5.3.
\section{Setting, assumptions, and notation}
Throughout, \(d \geq 2\) and all homology groups are taken with
coefficients in a fixed field, suppressed from the notation.
\(M \subset \mathbb{R}^{d}\) is the closure of a bounded, connected open
set; its homology groups are denoted \(H_{q}(M)\), \(q = 0,\ldots,d\).
(Connectedness avoids notational distractions only; for several
components the arguments apply componentwise after adding the
corresponding spectral-access terms.) For a compact set
\(A \subset \mathbb{R}^{d}\) and \(\rho > 0\),
\(A^{\rho} = \{ x \in \mathbb{R}^{d}:\operatorname{dist}(x,A) \leq \rho\}\) is its
ambient Euclidean \(\rho\)-offset, and \(d_{H}\) denotes Hausdorff
distance. The recovery events are
\[R_{q}(T,r) = \left\{ H_{q}\left( \Gamma_{T}^{r} \right) \cong H_{q}(M) \right\},\quad\quad R(T,r) = \underset{q = 0}{\bigcap^{d}}R_{q}(T,r).\]
\textbf{Geometric assumptions.} The deterministic reconstruction
arguments use the following standing hypotheses.
\begin{assumption}[Positive reach]
\(\operatorname{reach}(M) > 0\): every
\(z \in \mathbb{R}^{d}\) with \(0 < \operatorname{dist}(z,M) < \operatorname{reach}(M)\) has a unique
nearest point in \(M\). Positive reach is understood as a property of
the compact set \(M\) itself; it rules out cusps and near self-contacts
below the scale \(\operatorname{reach}(M)\) and guarantees a nontrivial interval of
offset radii on which the topology of offsets of \(M\) is stable.
\end{assumption}
For \(x,y \in M\), the intrinsic path metric is
\(d_{M}(x,y) = \inf_{\gamma}length(\gamma)\), the infimum over
rectifiable curves in \(M\) joining \(x\) to \(y\); the intrinsic ball
is \(B_{M}(x,\varepsilon) = \{ y \in M:d_{M}(x,y) < \varepsilon\}\) and
the Euclidean ball is \(B(x,\varepsilon)\). Always
\(|x - y| \leq d_{M}(x,y)\).
\begin{assumption}[Local intrinsic-extrinsic comparability]
There
exist \(c_{ie} \geq 1\) and \(\rho_{ie} > 0\) such that
\(d_{M}(x,y) \leq c_{ie}|x - y|\) whenever \(x,y \in M\) and
\(|x - y| < \rho_{ie}\).
\end{assumption}
Let \(\mu\) be normalized Lebesgue measure on \(M\),
\(\mu(A) = |A|/|M|\).
\begin{assumption}[Local volume regularity]
There exist
\(c_{\operatorname{vol}},C_{\operatorname{vol}} > 0\) and \(\rho_{\operatorname{vol}} > 0\) such that
\(c_{\operatorname{vol}}\,\varepsilon^{d} \leq \mu\left( B_{M}(x,\varepsilon) \right) \leq C_{\operatorname{vol}}\,\varepsilon^{d}\)
for every \(x \in M\) and \(0 < \varepsilon < \rho_{\operatorname{vol}}\).
\end{assumption}
Let \(N_{\varepsilon}(M)\) denote the minimal cardinality of an
intrinsic \(\varepsilon\)-net of \(M\), i.e.~the least \(N\) for which
there exist \(x_{1},\ldots,x_{N} \in M\) with
\(M \subseteq \bigcup_{i = 1}^{N}B_{M}\left( x_{i},\varepsilon \right)\).
A standard packing argument under Assumption 2.3 gives, for all
sufficiently small \(\varepsilon\),
\[c_{N}\,\varepsilon^{- d} \leq N_{\varepsilon}(M) \leq C_{N}\,\varepsilon^{- d},\quad\quad\text{hence}\quad\quad\log N_{\varepsilon}(M) \asymp d\log\varepsilon^{- 1}.\]
\textbf{Density.} For a compact set \(S \subset M\), its intrinsic
covering radius is \(\rho_{M}(S) = \sup_{x \in M}d_{M}(x,S)\); \(S\) is
\emph{intrinsically} \(\varepsilon\)\emph{-dense} in \(M\) if
\(\rho_{M}(S) \leq \varepsilon\), and \emph{extrinsically}
\(\varepsilon\)\emph{-dense} if
\(\sup_{x \in M}\operatorname{dist}(x,S) \leq \varepsilon\). Since \(S \subset M\),
extrinsic \(\varepsilon\)-density is exactly
\(d_{H}(S,M) \leq \varepsilon\). The following comparison is used
repeatedly.
\begin{lemma}[Density comparison]
If \(S \subset M\)
is intrinsically \(\varepsilon\)-dense in \(M\),
then it is extrinsically \(\varepsilon\)-dense. Conversely, if
\(S\) is extrinsically \(\varepsilon\)-dense and
\(\varepsilon < \rho_{ie}\), then \(S\) is intrinsically
\(c_{ie}\,\varepsilon\)-dense.
\end{lemma}
\begin{proof} The first implication is \(|x - y| \leq d_{M}(x,y)\). For
the converse, given \(x \in M\) choose \(y \in S\) with
\(|x - y| \leq \varepsilon < \rho_{ie}\); Assumption 2.2 gives
\(d_{M}(x,y) \leq c_{ie}\,\varepsilon\), and one takes the supremum over
\(x\). \end{proof}
\textbf{Analytic assumptions.}
\begin{assumption}[Admissible reflecting domain]
The boundary
regularity of \(M\) suffices to ensure: existence and uniqueness in law
of reflected Brownian motion in \(M\); existence of a symmetric Neumann
heat kernel; compactness of the Neumann semigroup on \(L^{2}(M,\mu)\);
discreteness of the Neumann spectrum with positive spectral gap; and a
uniform local chart structure --- bi-Lipschitz charts at a fixed
geometric scale, with boundary flattening near \(\partial M\) ---
sufficient for the condenser comparisons carried out in Appendix A. The
precise chart properties are listed in Appendix A.3; we emphasize that
the small-ball capacity estimates themselves are \emph{not} assumed but
are proved in Appendix A from this chart structure. A sufficient
condition is that \(M\) be a compact connected Lipschitz domain with
uniform interior and boundary chart constants {[}12, 13, 14{]}. Only the
listed consequences are used in the analytic and probabilistic
arguments. The deterministic reconstruction of Section 3 additionally
invokes the positive reach of Assumption 2.1, which is logically
independent of the chart structure above and is not implied by it: a
Lipschitz domain with a reentrant corner has zero reach, so it satisfies
the conditions listed here yet is excluded by Assumption 2.1.
Admissibility throughout requires both assumptions.
\end{assumption}
On \(L^{2}(M,\mu)\), consider the symmetric Dirichlet form
\[\mathcal{E}(f,g) = \frac{1}{2}\int_{M}^{}\nabla f \cdot \nabla g\, d\mu,\quad\quad f,g \in H^{1}(M).\]
The associated conservative Hunt process is reflected Brownian motion
\(X = \left( X_{t} \right)_{t \geq 0}\), reversible with invariant
probability measure \(\mu\); its generator is the Neumann Laplacian
\(\frac{1}{2}\Delta\) with \(\partial_{n}f = 0\) on \(\partial M\). The
Neumann spectrum is
\(0 = \lambda_{0}(M) < \lambda_{1}(M) \leq \lambda_{2}(M) \leq \cdots\),
and \(\lambda_{1}(M)\) is the spectral gap. Let \(p_{t}(x,y)\) be the
transition density with respect to \(\mu\).
\begin{assumption}[Heat-kernel domination after burn-in]
There
exist \(s_{0} > 0\) and \(a_{hk} \geq 1\) such that
\(p_{t}(x,y) \leq a_{hk}\) for all \(t \geq s_{0}\) and \(x,y \in M\);
equivalently
\(\mathbb{P}_{x}\left( X_{t} \in A \right) \leq a_{hk}\,\mu(A)\) for
every starting point, \(t \geq s_{0}\), and measurable \(A \subset M\).
\end{assumption}
On compact regular domains such a bound follows from standard Neumann
heat-kernel estimates together with spectral information {[}14{]}, and
one may typically take \(s_{0} \asymp \lambda_{1}(M)^{- 1}\). The
assumption is used only to pass from hitting estimates started in
stationarity to arbitrary starting points; the lower bounds of Section 7
show a corresponding spectral access term is generally unavoidable. Thus
the upper bounds carry \(s_{0}\), while the bottleneck lower bound of
Proposition 7.3 forces \(\lambda_{1}(M)^{- 1}\), and the two match only
through the identification \(s_{0} \asymp \lambda_{1}(M)^{- 1}\); this
holds in the regime above but is not asserted unconditionally, since
intermediate scales can make heat-kernel domination slower than the
spectral gap.

\textbf{Killed eigenvalues.} For a measurable target \(A \subset M\) of
positive capacity, the hitting time is
\(\tau_{A} = \inf\{ t \geq 0:X_{t} \in A\}\), the killed semigroup is
\(P_{t}^{A}f(x) = \mathbb{E}_{x}\left\lbrack f\left( X_{t} \right)\,\mathbf{1}_{\{\tau_{A} > t\}} \right\rbrack\),
and the principal killed eigenvalue is
\[\lambda_{A} = \inf\left\{ \frac{\mathcal{E}(f,f)}{\parallel f \parallel_{L^{2}(\mu)}^{2}}:f \in H^{1}(M),\ f = 0\text{ quasi-everywhere on }A,\ f \neq 0 \right\},\]
the bottom of the spectrum of the Neumann Laplacian on \(M\backslash A\)
with Dirichlet condition on \(A\). The smaller \(\lambda_{A}\), the
harder \(A\) is to hit.

\textbf{Capacity scale.} The small-target estimates are governed by the
Brownian capacity scale \(\kappa_{d}(\varepsilon)\) defined in Section
1.2. Note the contrast with volume:
\(\mu\left( B_{M}(x,\varepsilon) \right) \asymp \varepsilon^{d}\) under
Assumption 2.3, but the hitting scale proved in Section 5 is
\(\lambda_{B_{M}(x,\varepsilon)} \asymp \kappa_{d}(\varepsilon)\),
giving local detection times \(\log\varepsilon^{- 1}\) (\(d = 2\)) and
\(\varepsilon^{2 - d}\) (\(d \geq 3\)) rather than
\(\varepsilon^{- d}\).

\textbf{Constants.} Constants \(c,C,c_{i},C_{i}\) may change from line
to line and may depend on \(d\), \(M\), \(\operatorname{reach}(M)\),
\(\left( c_{ie},\rho_{ie} \right)\),
\(\left( c_{\operatorname{vol}},C_{\operatorname{vol}},\rho_{\operatorname{vol}} \right)\), the analytic
reflecting-domain constants, and \(\left( s_{0},a_{hk} \right)\) --- but
never on \(T\), \(\delta\), \(\varepsilon\), \(r\), the size \(m\) of a
target family, or the centers of sufficiently small target balls. We
write \(f(\varepsilon) \lesssim g(\varepsilon)\) if \(f \leq Cg\) for
small \(\varepsilon\), and \(f \asymp g\) if both \(f \lesssim g\) and
\(g \lesssim f\).
\section{Deterministic reconstruction and homological certificates}
The probabilistic estimates of Sections 4 and 5 will show that reflected
Brownian motion hits prescribed targets with high probability after a
time governed by killed eigenvalues. This section provides the
deterministic bridge from such hitting events to topological
conclusions.
\subsection{Offset reconstruction below the reach}
We use the following standard consequence of positive-reach stability;
it is a mild reformulation of the offset reconstruction results
associated with positive reach and weak feature size {[}1, 2, 3{]}.
\begin{theorem}[Offset reconstruction]
There exist
constants \(\varepsilon_{rec} > 0\), \(C_{low} > 0\),
\(C_{up} > 0\), depending only on the positive-reach
reconstruction constants of \(M\), such that the following holds.
Let \(S \subset M\) be compact with
\(d_{H}(S,M) \leq \varepsilon\) and
\(0 < \varepsilon < \varepsilon_{rec}\). Then, for every
reconstruction radius \(r\) with

\[C_{low}\,\varepsilon \leq r \leq C_{up}\, \operatorname{reach}(M),\]

the offset \(S^{r}\) is homotopy equivalent to
\(M\); in particular \(H_{q}\left( S^{r} \right) \cong H_{q}(M)\)
for every \(q = 0,\ldots,d\).
\end{theorem}
\begin{proof} Since \(d_{H}(S,M) \leq \varepsilon\), the offset
filtrations of \(S\) and \(M\) are \(\varepsilon\)-interleaved:
\(S^{\rho} \subseteq M^{\rho + \varepsilon}\) and
\(M^{\rho} \subseteq S^{\rho + \varepsilon}\) for every \(\rho \geq 0\);
in particular
\(M^{r - \varepsilon} \subseteq S^{r} \subseteq M^{r + \varepsilon}\).
We emphasize that this offset sandwich alone does not yield the
conclusion by elementary algebraic topology --- a space squeezed between
two homotopy-equivalent offsets need not share their homotopy type ---
so we invoke the offset reconstruction theorem of Chazal, Cohen-Steiner
and Lieutier {[}3{]}, stated next.
Recall that the \emph{weak feature size} \(wfs(K)\) of a compact set
\(K \subset \mathbb{R}^{n}\) is the infimum of the positive critical
values of the distance function \(d_{K}\) in the sense of the
generalized gradient of {[}3{]} (with \(wfs(K) = + \infty\) when
\(d_{K}\) has no positive critical value). The reconstruction theorem of
{[}3{]} asserts: \emph{if} \(K,P \subset \mathbb{R}^{n}\) \emph{are
compact with} \(d_{H}(K,P) \leq \varepsilon\) \emph{and}
\(\varepsilon < \frac{1}{4}wfs(K)\)\emph{, then for every radius} \(r\)
\emph{with} \(2\varepsilon \leq r \leq wfs(K) - 2\varepsilon\) \emph{the
offset} \(P^{r}\) \emph{is homotopy equivalent to} \(K^{\eta}\)
\emph{for every sufficiently small} \(\eta > 0\)\emph{.} The mechanism
is the critical-point theory of distance functions: on
\(\left( 0,wfs(K) \right)\) the function \(d_{K}\) has no critical
value, so by the Isotopy Lemma the offsets \(K^{s}\),
\(0 < s < wfs(K)\), share one homotopy type, and the interleaving above
--- together with the matching one for \(P\) --- identifies the homotopy
type of \(P^{r}\) with that of \(K^{s}\); see {[}3{]} for the homotopy
argument, {[}1{]} for the tubular-neighbourhood structure of the stable
regime, and {[}2{]} for the Riemannian-submanifold case.
We verify the hypotheses with \(K = M\) and \(P = S\). First,
\(wfs(M) \geq \operatorname{reach}(M)\): on the open tube
\(\{ x:0 < d_{M}(x) < \operatorname{reach}(M)\}\) the nearest-point projection onto
\(M\) is single-valued by Assumption 2.1, so \(d_{M}\) is continuously
differentiable there with \(\left| \nabla d_{M} \right| \equiv 1\);
hence \(d_{M}\) has no critical value in \(\left( 0,\operatorname{reach}(M) \right)\),
and its smallest positive critical value is at least \(\operatorname{reach}(M)\). In
particular \(wfs(M) \geq \operatorname{reach}(M) > 0\). Moreover, having positive
reach, \(M\) is a Lipschitz neighbourhood retract, hence an ANR, and by
Federer's tubular-neighbourhood theorem {[}1{]} the offset \(M^{\eta}\)
deformation retracts onto \(M\) for every \(0 \leq \eta < \operatorname{reach}(M)\), so
\(M^{\eta} \simeq M\) for small \(\eta\). Now set
\(\varepsilon_{rec} = \frac{1}{4}\operatorname{reach}(M)\), \(C_{low} = 2\) and
\(C_{up} = \frac{1}{2}\). For \(0 < \varepsilon < \varepsilon_{rec}\)
and \(C_{low}\varepsilon \leq r \leq C_{up}\, \operatorname{reach}(M)\) one has
\(\varepsilon < \frac{1}{4}\operatorname{reach}(M) \leq \frac{1}{4}wfs(M)\) and
\(2\varepsilon \leq r \leq \frac{1}{2}\operatorname{reach}(M) \leq \operatorname{reach}(M) - 2\varepsilon \leq wfs(M) - 2\varepsilon\),
so the reconstruction theorem applies and gives
\(S^{r} \simeq M^{\eta} \simeq M\). This is the asserted homotopy
equivalence, and the homology isomorphisms
\(H_{q}\left( S^{r} \right) \cong H_{q}(M)\), \(q = 0,\ldots,d\),
follow.
Finally, since the conclusion is a genuine homotopy equivalence, the
isomorphism \(H_{q}\left( S^{r} \right) \cong H_{q}(M)\) holds over any
coefficient ring --- in particular over \(\mathbb{Z}\), torsion included
--- and so does not rely on the standing convention of Section 2 that
homology is taken with coefficients in a fixed field. That convention
enters only where the Nerve theorem and the stability of persistence
diagrams are used (Section 8). \end{proof}
The precise constants are immaterial; what matters is that for
\(\varepsilon\) small relative to \(\operatorname{reach}(M)\) there is a nonempty
window of admissible radii --- thick enough to fill sampling gaps, below
the scale at which genuine features of \(M\) are destroyed. Combining
Theorem 3.1 with Lemma 2.4 gives the intrinsic form used later.
\begin{corollary}[Reconstruction from intrinsic density]
There exist constants \(\varepsilon_{int} > 0\),
\(C_{low}' > 0\), \(C_{up}' > 0\), depending only on the
geometry of \(M\), such that if \(S \subset M\) is compact
and intrinsically \(\varepsilon\)-dense in \(M\) with
\(0 < \varepsilon < \varepsilon_{int}\), then
\(H_{q}\left( S^{r} \right) \cong H_{q}(M)\) for every \(q\)
and every \(r\) with
\(C_{low}'\,\varepsilon \leq r \leq C_{up}'\, \operatorname{reach}(M)\). In
subsequent statements derived from this corollary, the primes are
dropped and the window constants are renamed \(C_{low},C_{up}\)
after each adjustment.
\end{corollary}
\subsection{Homological certificates}
Full density is sufficient but may be excessive: a path can recover a
given feature without becoming dense in the whole domain. We formalize
this.
For a finite family \(\mathcal{A = \{}A_{1},\ldots,A_{m}\}\) of
measurable subsets of \(M\) and a continuous path
\(\gamma:\lbrack 0,T\rbrack \rightarrow M\) with trace
\(\operatorname{Tr}(\gamma) = \gamma\left( \lbrack 0,T\rbrack \right)\), we say
\(\gamma\) \emph{hits all targets} in \(\mathcal{A}\) if
\(\operatorname{Tr}(\gamma) \cap A_{i} \neq \varnothing\) for every \(i\). For
reflected Brownian motion this is the event
\[H_{T}\left( \mathcal{A} \right) = \underset{i = 1}{\bigcap^{m}}\{\tau_{A_{i}} \leq T\}.\]
\begin{definition}[Homological certificate]
Let
\(q \in \{ 0,\ldots,d\}\) and \(r > 0\). A finite family
\(\mathcal{A = \{}A_{1},\ldots,A_{m}\}\) of measurable subsets of \(M\)
is a \emph{degree-}\(q\)\emph{, radius-}\(r\) \emph{homological
certificate} for \(M\) if every continuous path
\(\gamma:\lbrack 0,T\rbrack \rightarrow M\) hitting all targets in
\(\mathcal{A}\) satisfies
\(H_{q}\left( \operatorname{Tr}(\gamma)^{r} \right) \cong H_{q}(M)\). If the
implication holds simultaneously for every \(q = 0,\ldots,d\), then
\(\mathcal{A}\) is a \emph{full} homological certificate at radius
\(r\).
\end{definition}
The definition is deliberately deterministic: a certificate asserts that
certain visits are topologically sufficient, and probability enters only
when we estimate the time needed to realize them. In a planar annulus,
for instance, a certificate for \(H_{1}\) may consist of finitely many
witness regions arranged around the central obstacle: if the thickened
trace contains a loop winding around the obstacle, that class is
realized in the offset even though the trace is far from covering the
annulus. As discussed after Proposition 3.4, this is a detection
guarantee --- it forces the class to be present, but not the two-sided
isomorphism of Definition 3.3, which also excludes spurious cycles.
\begin{figure}[htbp]
\centering
\robustfig[0.70\linewidth]{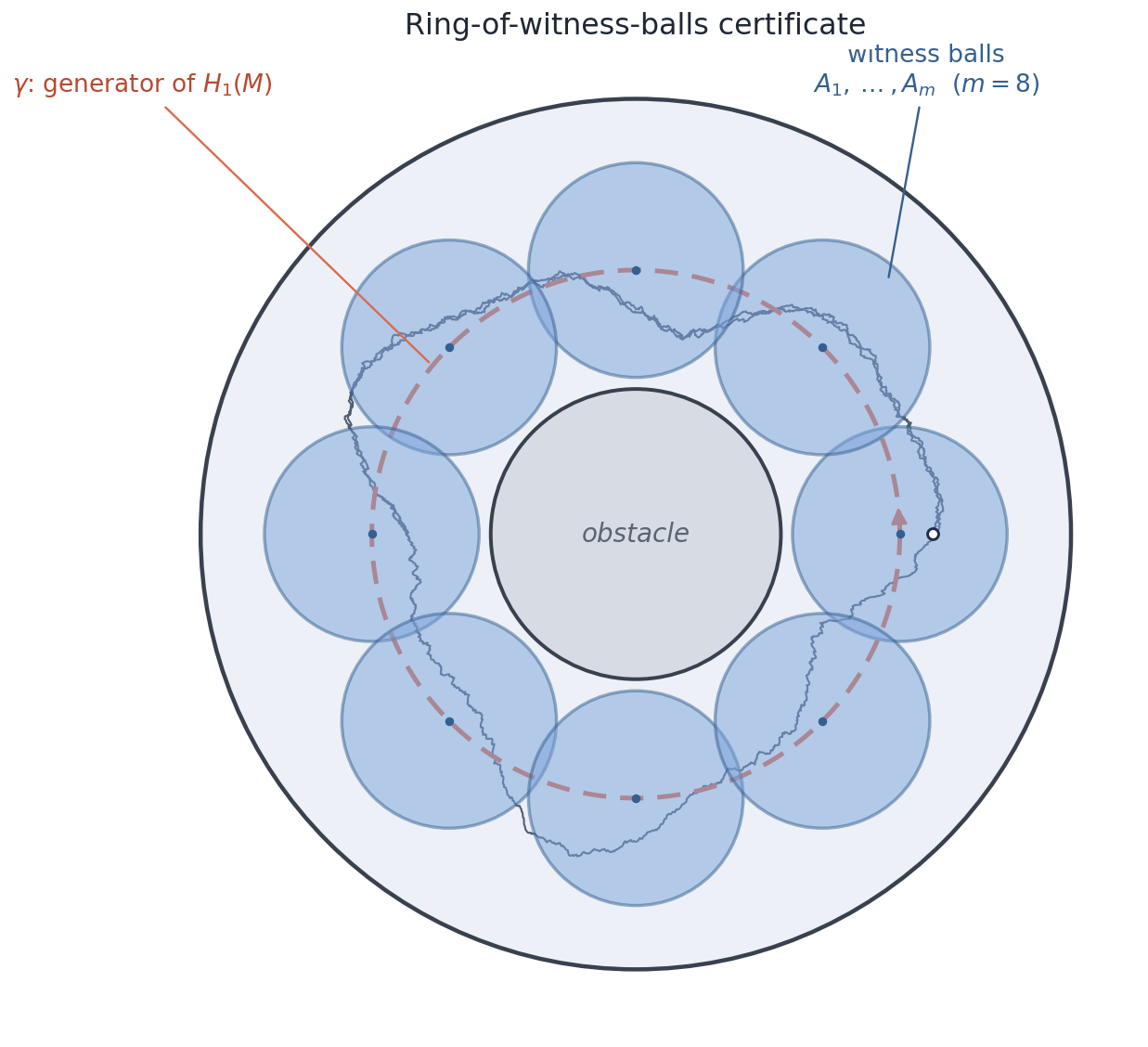}
\caption{A ring-of-witness-balls certificate for detection. To certify the generator \(\gamma\) of \(H_{1}(M)\) around an obstacle --- here a planar annulus --- it suffices to place a bounded ring of \(m\) witness balls \(A_{1},\ldots,A_{m}\) along \(\gamma\). Any reflected trajectory that meets every ball (one such trace is shown) forces its \(r\)-offset \(\Gamma_{T}^{\, r}\) to contain a loop winding around the obstacle, so the class is realized in the offset and the inclusion \(H_{q}\left( \Gamma_{T}^{\, r} \right) \rightarrow H_{q}(M)\) is surjective onto it. This is \emph{detection}: it certifies that the feature is present without the two-sided isomorphism of Definition 3.3, and it costs only \(\log m\) in the hitting time, in place of the covering logarithm \(\log N_{\varepsilon}(M)\) of a full intrinsic \(\varepsilon\)-net. Spurious cycles are not excluded, so the guarantee is one-sided.}
\label{fig:2}
\end{figure}
\begin{proposition}[Certificates: recovery and existence]
(i) If \(\mathcal{A}\) is a degree-\(q\),
radius-\(r\) certificate, then
\(H_{T}\left( \mathcal{A} \right) \subseteq R_{q}(T,r)\); if
\(\mathcal{A}\) is full, then
\(H_{T}\left( \mathcal{A} \right) \subseteq R(T,r)\).

(ii) There exist constants \(C_{low},C_{up},\varepsilon_{0} > 0\)
such that for every \(0 < \varepsilon < \varepsilon_{0}\)
and every intrinsic \(\varepsilon\)-net
\(\{ x_{1},\ldots,x_{N}\}\) of \(M\) with
\(N = N_{\varepsilon}(M)\), the family

\[\mathcal{A}_{\varepsilon} = \{ B_{M}\left( x_{1},\varepsilon \right),\ldots,B_{M}\left( x_{N},\varepsilon \right)\}\]

is a full homological certificate at every radius \(r\)
with
\(C_{low}\,\varepsilon \leq r \leq C_{up}\, \operatorname{reach}(M)\). In
particular, full certificates of size \(N_{\varepsilon}(M)\)
exist.
\end{proposition}
\begin{proof} (i) On \(H_{T}\left( \mathcal{A} \right)\) the path
\(X|_{\lbrack 0,T\rbrack}\) hits every target, so the certificate
property applied to \(\gamma = X|_{\lbrack 0,T\rbrack}\), whose trace is
\(\Gamma_{T}\), gives
\(H_{q}\left( \Gamma_{T}^{r} \right) \cong H_{q}(M)\); the full case
follows by applying the implication for each \(q\).
\begin{enumerate}
\def\labelenumi{(\roman{enumi})}
\setcounter{enumi}{1}
\tightlist
\item
  Suppose a continuous path \(\gamma\) hits every ball of
  \(\mathcal{A}_{\varepsilon}\). For any \(x \in M\) choose \(x_{i}\)
  with \(d_{M}\left( x,x_{i} \right) < \varepsilon\) and \(t\) with
  \(d_{M}\left( \gamma(t),x_{i} \right) < \varepsilon\); then
  \(d_{M}\left( x,\gamma(t) \right) < 2\varepsilon\), so
  \(\rho_{M}\left( \operatorname{Tr}(\gamma) \right) \leq 2\varepsilon\). Corollary 3.2
  applied at scale \(2\varepsilon\) gives
  \(H_{q}\left( \operatorname{Tr}(\gamma)^{r} \right) \cong H_{q}(M)\) for all \(q\)
  and all \(r\) in the window, after absorbing the factor \(2\) into the
  constants. \end{enumerate}\end{proof}
\begin{remark}[detection versus reconstruction]
The isomorphism
demanded in Definition 3.3 is two-sided, and Proposition 3.4(ii) is the
only construction here that supplies it: a full intrinsic net forces
every admissible trace to be \(\varepsilon\)-dense, and density makes
\(\Gamma_{T}^{r}\) homotopy equivalent to \(M\) by Theorem 3.1. A
strictly smaller family is, in general, not a certificate in this sense,
and the obstruction is injectivity. A path may hit every target of a
sub-net family yet leave a region of \(M\) of diameter exceeding \(r\)
uncovered; the offset \(\Gamma_{T}^{r}\) then carries a \(q\)-cycle that
bounds in \(M^{r}\) but not in \(\Gamma_{T}^{r}\), so the
inclusion-induced map
\(\iota_{*}:H_{q}\left( \Gamma_{T}^{r} \right) \rightarrow H_{q}(M)\)
--- factoring through \(H_{q}\left( M^{r} \right) \cong H_{q}(M)\) ---
acquires a kernel and the isomorphism fails. What a small family can
force is surjectivity of \(\iota_{*}\): a ring of witness regions
encircling the obstacle of a planar annulus, at mutual spacing below
\(2r\), guarantees that the thickened trace contains a loop mapping to a
generator of \(H_{1}(M)\), so the genuine class is realized in the
offset, and more generally a family witnessing representatives of a
generating set of \(H_{q}(M)\) makes \(\iota_{*}\) onto. We call this
one-sided guarantee detection: it forces the prescribed classes to be
realized but does not exclude spurious ones. The estimate of the next
section is blind to the distinction --- it bounds only the time to visit
a finite family --- so the feature-adaptive scale of Theorem 6.1 governs
detection by a small certificate exactly as it governs reconstruction by
the net: one hitting bound with two deterministic readings. Whether a
family strictly smaller than a full net can force the two-sided
isomorphism in an intermediate degree \(0 < q < d\) is left open; the
extreme degrees are degenerate, since \(H_{0}\) is fixed by the
connectedness of the trace and \(H_{d}(M) = 0\) for
\(M \subseteq \mathbb{R}^{d}\).
\end{remark}
\begin{definition}[Witness family; detectable and robust
features]
Let \(q \in \{ 0,\ldots,d\}\), let \(r > 0\), and let
\(G \subseteq H_{q}(M)\) be a subspace. A finite family
\(\mathcal{W} = \{ W_{1},\ldots,W_{m}\}\) of measurable subsets of \(M\)
is a \emph{degree-}\(q\)\emph{, radius-}\(r\) \emph{witness family} for
\(G\) if every continuous path
\(\gamma:\lbrack 0,T\rbrack \rightarrow M\) hitting all of
\(W_{1},\ldots,W_{m}\) satisfies \(G \subseteq im\iota_{*}\), where
\(\iota_{*}:H_{q}\left( \operatorname{Tr}(\gamma)^{r} \right) \rightarrow H_{q}(M)\) is
induced by inclusion. Taking \(G = H_{q}(M)\) recovers surjectivity of
\(\iota_{*}\), the detection guarantee of the Remark above. A
degree-\(q\) class --- and the \emph{feature} it represents --- is
\emph{detectable} at scale \(\varepsilon\) by \(m\) witnesses if the
line it spans admits such a witness family of cardinality \(m\)
consisting of intrinsic balls \(B_{M}\left( x_{i},\varepsilon \right)\);
the feature is \emph{robust} if \(m\) may be taken bounded independently
of the covering number \(N_{\varepsilon}(M)\).
\end{definition}
A witness family is the one-sided counterpart of the certificate of
Definition 3.3: it forces the prescribed classes to be realised in the
offset but, unlike a certificate, does not exclude spurious ones.
Bounded witness families are exhibited here only for specific features
--- the generator of \(H_{1}\) around an obstacle in a planar domain, by
the ring of witness regions of Figure 2 and the Remark above; whether a
given feature is robust in general, equivalently whether the detection
rate of Section 6 is attained beyond such examples, is left open, in the
same spirit as the sub-net reconstruction problem.
Thus density-based reconstruction is the extremal special case of
certificate-based reconstruction: in the universal theory the
certificate is a full net of size \(\asymp N_{\varepsilon}(M)\), while
the feature-adaptive theory seeks smaller certificates. The stochastic
cost of a certificate, however, is not determined by its cardinality
alone: a small certificate may contain targets that are hard to hit. The
correct stochastic measurement is the spectral profile introduced next.
\section{Spectral hitting estimates for finite target families}
This section proves the probabilistic estimate converting target
visitation into quantitative bounds. It is stated for arbitrary finite
families; in Section 6 it is applied both to full intrinsic nets and to
homological certificates --- the same estimate, with two deterministic
meanings.
\begin{lemma}[Survival from stationarity]
For every
compact \(A \subset M\) of positive capacity and every
\(t \geq 0\),

\[\mathbb{P}_{\mu}\left( \tau_{A} > t \right) \leq e^{- \lambda_{A}t},\]

where \(\mathbb{P}_{\mu}\) is the law of reflected Brownian
motion started from \(\mu\).
\end{lemma}
\begin{proof} By reversibility, the killed semigroup \(P_{t}^{A}\) is
self-adjoint on \(L^{2}(M,\mu)\) with
\(\parallel P_{t}^{A} \parallel_{2 \rightarrow 2} = e^{- \lambda_{A}t}\).
Since the process starts from \(\mu\),
\[\mathbb{P}_{\mu}\left( \tau_{A} > t \right) = \int_{M}^{}P_{t}^{A}\mathbf{1}(x)\, d\mu(x) = \langle\mathbf{1},\, P_{t}^{A}\mathbf{1}\rangle_{L^{2}(\mu)} \leq \parallel P_{t}^{A} \parallel_{2 \rightarrow 2}\, \parallel \mathbf{1} \parallel_{2}^{2} = e^{- \lambda_{A}t},\]
using \(\parallel \mathbf{1} \parallel_{2}^{2} = \mu(M) = 1\). \end{proof}
\begin{lemma}[Burn-in survival estimate]
Under Assumption
2.6, for every target \(A \subset M\) of positive capacity, every
starting point \(x \in M\), and every \(T \geq s_{0}\),

\[\mathbb{P}_{x}\left( \tau_{A} > T \right) \leq a_{hk}\, e^{- \lambda_{A}\left( T - s_{0} \right)}.\]
\end{lemma}
\begin{proof} By the Markov property at time \(s_{0}\) and dropping an
indicator,
\begin{align*}
\mathbb{P}_{x}\left( \tau_{A} > T \right) &\leq \mathbb{E}_{x}\left\lbrack \mathbb{P}_{X_{s_{0}}}\left( \tau_{A} > T - s_{0} \right) \right\rbrack\\
&= \int_{M}^{}\mathbb{P}_{y}\left( \tau_{A} > T - s_{0} \right)\, p_{s_{0}}(x,y)\, d\mu(y).
\end{align*}
Assumption 2.6 bounds \(p_{s_{0}}(x,y) \leq a_{hk}\), so the right-hand
side is at most
\(a_{hk}\,\mathbb{P}_{\mu}\left( \tau_{A} > T - s_{0} \right)\), and
Lemma 4.1 concludes. \end{proof}
\begin{definition}[Spectral profile]
For a finite family
\(\mathcal{A = \{}A_{1},\ldots,A_{m}\}\) of positive-capacity targets,
set

\[\Phi_{\mathcal{A}}(u) = \sum_{A\mathcal{\in A}}^{}e^{- \lambda_{A}u},\quad\quad u \geq 0.\]
\end{definition}
The profile records both the number of targets and their detectability:
a target with small \(\lambda_{A}\) contributes a slowly decaying term.
\begin{theorem}[Finite-target spectral hitting estimate]
Under Assumption 2.6, let
\(\mathcal{A = \{}A_{1},\ldots,A_{m}\}\) be a finite family of
measurable subsets of \(M\), each of positive capacity. Then, for
every \(x \in M\) and every \(T \geq s_{0}\),

\[\mathbb{P}_{x}\left( H_{T}\left( \mathcal{A} \right)^{c} \right) \leq a_{hk}\,\Phi_{\mathcal{A}}\left( T - s_{0} \right).\]
\end{theorem}
\begin{proof}
\(H_{T}\left( \mathcal{A} \right)^{c} = \bigcup_{A\mathcal{\in A}}\{\tau_{A} > T\}\);
the union bound and Lemma 4.2 give
\[\mathbb{P}_{x}\left( H_{T}\left( \mathcal{A} \right)^{c} \right) \leq \sum_{A\mathcal{\in A}}^{}\mathbb{P}_{x}\left( \tau_{A} > T \right) \leq a_{hk}\sum_{A\mathcal{\in A}}^{}e^{- \lambda_{A}\left( T - s_{0} \right)}. \]\end{proof}
The union bound is elementary; the substance is that the correct decay
rates are killed eigenvalues --- and, for small balls, capacities
(Section 5).
\begin{corollary}[Uniform spectral hitting bound]
If
\(\lambda_{A} \geq \lambda_{\min} > 0\) for every
\(A\mathcal{\in A}\), then
\(\mathbb{P}_{x}\left( H_{T}\left( \mathcal{A} \right)^{c} \right) \leq a_{hk}\,\left| \mathcal{A} \right|\, e^{- \lambda_{\min}\left( T - s_{0} \right)}\)
for every \(x \in M\) and \(T \geq s_{0}\);
consequently
\(\mathbb{P}_{x}\left( H_{T}\left( \mathcal{A} \right) \right) \geq 1 - \delta\)
whenever

\[T \geq s_{0} + \lambda_{\min}^{- 1}\left( \log\left| \mathcal{A} \right| + \log\frac{a_{hk}}{\delta} \right).\]
\end{corollary}
\begin{proof} The profile bound
\(\Phi_{\mathcal{A}}\left( T - s_{0} \right) \leq \left| \mathcal{A} \right|\, e^{- \lambda_{\min}\left( T - s_{0} \right)}\)
and Theorem 4.4 give the first estimate; making the right-hand side at
most \(\delta\) and taking logarithms gives the time condition. \end{proof}
\section{Small-hole eigenvalues and reflected capacity}
Corollary 4.5 reduces explicit recovery times to lower bounds on the
killed eigenvalues \(\lambda_{B_{M}(x,\varepsilon)}\), uniformly over
all centers \(x \in M\) --- including centers near the reflecting
boundary, which any intrinsic net necessarily contains. This section
proves that the correct scale is Brownian capacity:
\[\lambda_{B_{M}(x,\varepsilon)} \asymp \kappa_{d}(\varepsilon)\quad\text{uniformly in }x \in M.\]
Only the lower bound is needed for the upper recovery estimates; the
upper bound shows in Section 7 that the inverse-capacity term is also
unavoidable. The proof has two parts: a comparison of killed eigenvalues
with a Neumann capacity, and a uniform small-ball estimate for that
capacity, whose local condenser and boundary-flattening ingredients are
collected in Appendix A.
\subsection{Neumann capacity and the eigenvalue comparison}
Because \(M\) is compact and reflecting, the relevant capacity is not
the capacity in all of \(\mathbb{R}^{d}\) but a Neumann capacity adapted
to the compact domain. For compact \(A \subset M\), define
\begin{multline*}
\operatorname{Cap}_{N}(A) = \inf\Big\{ \mathcal{E}(u,u):u \in H^{1}(M),\ u \geq 1\text{ quasi-everywhere on }A,\\
\int_{M}^{}u\, d\mu = 0 \Big\}.
\end{multline*}
The mean-zero constraint removes the constants from the Neumann problem;
without it, a compact reflecting domain has zero-energy constants and
small sets would have trivial capacity. The normalization is immaterial
up to multiplicative constants.
\begin{proposition}[Capacity-eigenvalue comparison]
There
exist constants \(c_{ce},C_{ce} > 0\) and
\(\eta_{ce} > 0\), depending only on the analytic geometry of
\(M\), such that if \(A \subset M\) is compact with
\(\operatorname{Cap}_{N}(A) \leq \eta_{ce}\), then

\[c_{ce}\operatorname{Cap}_{N}(A) \leq \lambda_{A} \leq C_{ce}\operatorname{Cap}_{N}(A).\]

Moreover, in the small-capacity regime,
\(\lambda_{A} = \operatorname{Cap}_{N}(A)\left( 1 + O\left( \operatorname{Cap}_{N}(A) \right) \right)\)
after fixing the normalization of capacity consistently with
\(\mu\); this first-order form is proved below.
\end{proposition}
The comparison is standard in spectral perturbation theory for small
absorbing sets, in the tradition of Rauch-Taylor {[}18{]} and Ozawa
{[}19{]}. Because the paper relies on it at every scale, we give a
complete proof of the two-sided comparability, with explicit constants;
for definiteness one may take \(C_{ce} = 8\), \(c_{ce} = 1/4\), and
\(\eta_{ce} = \lambda_{1}(M)/32\). Throughout the proof,
\(\lambda_{1}(M)\) enters through the Poincaré inequality
\[\left. \parallel f - \int_{M}^{}f\, d\mu \right.\parallel_{L^{2}(\mu)}^{2} \leq \lambda_{1}(M)^{- 1}\,\mathcal{E}(f,f),\quad\quad f \in H^{1}(M),\]
which is the variational characterization of the Neumann spectral gap.
\begin{proof} \emph{Upper bound:} \(\lambda_{A} \leq 8\operatorname{Cap}_{N}(A)\). By
the definition of \(\operatorname{Cap}_{N}(A)\) as an infimum, choose an approximate
capacitary potential \(u_{A} \in H^{1}(M)\) with
\[u_{A} \geq 1\text{ quasi-everywhere on }A,\quad\quad\int_{M}^{}u_{A}\, d\mu = 0,\quad\quad\mathcal{E}\left( u_{A},u_{A} \right) \leq 2\operatorname{Cap}_{N}(A).\]
Define the test function
\[f_{A} = \left( 1 - u_{A} \right)_{+} = \max\left( 1 - u_{A},\, 0 \right).\]
Then \(f_{A} \in H^{1}(M)\), and \(f_{A} = 0\) quasi-everywhere on \(A\)
because \(u_{A} \geq 1\) there; thus \(f_{A}\) is admissible for the
killed variational problem, provided it is not the zero function, which
we verify below. Since taking the positive part is a normal contraction,
the Markovian property of the Dirichlet form gives
\[\mathcal{E}\left( f_{A},f_{A} \right)\mathcal{\leq E}\left( 1 - u_{A},\, 1 - u_{A} \right)\mathcal{= E}\left( u_{A},u_{A} \right) \leq 2\operatorname{Cap}_{N}(A).\]
For the \(L^{2}\) lower bound on \(f_{A}\), write
\(1 - u_{A} = \left( 1 - u_{A} \right)_{+} - \left( 1 - u_{A} \right)_{-}\)
and note
\(\left( 1 - u_{A} \right)_{-} = \left( u_{A} - 1 \right)_{+} \leq \left| u_{A} \right|\)
pointwise, so
\(\parallel \left( 1 - u_{A} \right)_{-} \parallel_{L^{2}(\mu)} \leq \parallel u_{A} \parallel_{L^{2}(\mu)}\).
Since \(u_{A}\) has mean zero, the Poincaré inequality gives
\[\parallel u_{A} \parallel_{L^{2}(\mu)}^{2} \leq \lambda_{1}(M)^{- 1}\,\mathcal{E}\left( u_{A},u_{A} \right) \leq \frac{2\operatorname{Cap}_{N}(A)}{\lambda_{1}(M)} \leq \frac{2\,\eta_{ce}}{\lambda_{1}(M)} = \frac{1}{16},\]
using \(\operatorname{Cap}_{N}(A) \leq \eta_{ce} = \lambda_{1}(M)/32\). Hence, by the
triangle inequality and
\(\parallel \mathbf{1} \parallel_{L^{2}(\mu)} = 1\),
\[\parallel f_{A} \parallel_{L^{2}(\mu)} \geq \parallel 1 - u_{A} \parallel_{L^{2}(\mu)} - \parallel \left( 1 - u_{A} \right)_{-} \parallel_{L^{2}(\mu)} \geq 1 - 2 \parallel u_{A} \parallel_{L^{2}(\mu)} \geq 1 - \frac{2}{4} = \frac{1}{2}.\]
In particular \(f_{A} \neq 0\), and the Rayleigh quotient yields
\[\lambda_{A} \leq \frac{\mathcal{E}\left( f_{A},f_{A} \right)}{\parallel f_{A} \parallel_{L^{2}(\mu)}^{2}} \leq \frac{2\operatorname{Cap}_{N}(A)}{1/4} = 8\operatorname{Cap}_{N}(A).\]
\emph{Lower bound:} \(\lambda_{A} \geq \frac{1}{4}\operatorname{Cap}_{N}(A)\). Let
\(\varphi_{A}\) be the principal eigenfunction of the killed problem,
chosen positive and normalized:
\begin{multline*}
\varphi_{A} \geq 0,\quad\quad\varphi_{A} = 0\text{ quasi-everywhere on }A,\\
\parallel \varphi_{A} \parallel_{L^{2}(\mu)} = 1,\quad\quad\mathcal{E}\left( \varphi_{A},\varphi_{A} \right) = \lambda_{A},
\end{multline*}
and let \(m_{A} = \int_{M}^{}\varphi_{A}\, d\mu \geq 0\) denote its
mean. The Poincaré inequality applied to \(\varphi_{A}\) gives
\[\parallel \varphi_{A} - m_{A} \parallel_{L^{2}(\mu)}^{2} \leq \lambda_{1}(M)^{- 1}\,\mathcal{E}\left( \varphi_{A},\varphi_{A} \right) = \frac{\lambda_{A}}{\lambda_{1}(M)}.\]
We first record that \(\lambda_{A}\) is small in the regime considered:
by the upper bound just proved,
\[\lambda_{A} \leq 8\operatorname{Cap}_{N}(A) \leq 8\,\eta_{ce} = \frac{\lambda_{1}(M)}{4}.\]
Consequently, by the triangle inequality in \(L^{2}(\mu)\) and the
normalization \(\parallel \varphi_{A} \parallel_{L^{2}(\mu)} = 1\),
\[m_{A} = \parallel m_{A}\mathbf{1} \parallel_{L^{2}(\mu)} \geq \parallel \varphi_{A} \parallel_{L^{2}(\mu)} - \parallel \varphi_{A} - m_{A} \parallel_{L^{2}(\mu)} \geq 1 - \left( \frac{\lambda_{A}}{\lambda_{1}(M)} \right)^{1/2} \geq 1 - \frac{1}{2} = \frac{1}{2},\]
so the mean of the eigenfunction is bounded away from zero. Now define
the capacitary competitor
\[u_{A} = 1 - \frac{\varphi_{A}}{m_{A}}.\]
This function is exactly admissible for the Neumann capacity of \(A\):
on \(A\), \(\varphi_{A} = 0\) quasi-everywhere, so \(u_{A} = 1\)
quasi-everywhere on \(A\); and its mean vanishes identically,
\[\int_{M}^{}u_{A}\, d\mu = 1 - \frac{1}{m_{A}}\int_{M}^{}\varphi_{A}\, d\mu = 1 - \frac{m_{A}}{m_{A}} = 0,\]
so no further adjustment is needed. Its energy is
\[\mathcal{E}\left( u_{A},u_{A} \right) = \frac{1}{m_{A}^{2}}\,\mathcal{E}\left( \varphi_{A},\varphi_{A} \right) = \frac{\lambda_{A}}{m_{A}^{2}} \leq 4\lambda_{A},\]
using \(m_{A} \geq 1/2\). Taking the infimum in the definition of
capacity over admissible functions gives
\[\operatorname{Cap}_{N}(A)\mathcal{\leq E}\left( u_{A},u_{A} \right) \leq 4\lambda_{A},\quad\quad\text{i.e.}\quad\quad\lambda_{A} \geq \frac{1}{4}\operatorname{Cap}_{N}(A).\]
This completes the two-sided comparison with the announced constants.
The sharper asymptotic form
\(\lambda_{A} = \operatorname{Cap}_{N}(A)\left( 1 + O\left( \operatorname{Cap}_{N}(A) \right) \right)\)
is established here, using only the Neumann spectral gap.
Recall
\(\lambda(A) = \inf\{\mathcal{E}(u,u):u \in H^{1}(M),\, u = 0\text{ q.e. on }A,\, \parallel u \parallel_{L^{2}(\mu)} = 1\}\),
and let \(\lambda_{1}(M) > 0\) be the Neumann spectral gap. For the
upper bound, \(1 - w_{A}\) vanishes q.e. on \(A\) with
\(\mathcal{E}\left( 1 - w_{A},1 - w_{A} \right) = \operatorname{Cap}_{N}(A)\) and
\(\parallel 1 - w_{A} \parallel_{L^{2}(\mu)}^{2} = 1 - O\left( \operatorname{Cap}_{N}(A) \right)\)
(using the gap on the mean-zero part of \(w_{A}\)), so
\(\lambda(A) \leq \operatorname{Cap}_{N}(A)\left( 1 + O\left( \operatorname{Cap}_{N}(A) \right) \right)\).
For the lower bound, let \(\varphi\) be the principal eigenfunction,
\(\parallel \varphi \parallel_{L^{2}(\mu)} = 1\), mean
\(\bar{\varphi}\). Then
\(\psi = \left( \bar{\varphi} - \varphi \right)/\bar{\varphi}\) equals
\(1\) q.e. on \(A\) with mean zero, so
\(\mathcal{E}(\varphi,\varphi)/{\bar{\varphi}}^{\, 2}\mathcal{= E}(\psi,\psi) \geq \operatorname{Cap}_{N}(A)\),
i.e.~\(\lambda(A) \geq {\bar{\varphi}}^{\, 2}\operatorname{Cap}_{N}(A)\); the gap
gives
\({\bar{\varphi}}^{\, 2} \geq 1 - \lambda_{1}(M)^{- 1}\lambda(A)\), and
combining yields the relation. \end{proof}
The heuristic content is worth recording: the first killed eigenfunction
is nearly constant away from a small target, and killing forces this
almost-constant function to dip to zero near \(A\); the energetic price
of the dip is precisely the capacity of \(A\).
\subsection{Uniform small-ball capacity}
\begin{proposition}[Uniform Neumann capacity of small intrinsic
balls]
There exist constants \(c_{\lg},C_{\lg} > 0\) and
\(\varepsilon_{\lg} > 0\) such that, for every \(x \in M\)
and every \(0 < \varepsilon < \varepsilon_{\lg}\),

\[c_{\lg}\,\kappa_{d}(\varepsilon) \leq \operatorname{Cap}_{N}\left( B_{M}(x,\varepsilon) \right) \leq C_{\lg}\,\kappa_{d}(\varepsilon).\]
\end{proposition}
The proof compares the global capacity with a local condenser problem.
Fix \(x \in M\) and a radius \(\rho > 0\) independent of
\(\varepsilon\), and let \(\operatorname{Cap}_{loc}(x,\varepsilon;\rho)\) denote the
minimal energy of a function equal to \(1\) on \(B_{M}(x,\varepsilon)\)
and to \(0\) on \(\partial B_{M}(x,\rho) \cap M\), with Neumann
condition on the reflecting part of \(\partial M\). Appendix A
(Proposition A.3) shows by bi-Lipschitz flattening and comparison with
Euclidean model condensers that
\[\operatorname{Cap}_{loc}(x,\varepsilon;\rho) \asymp \kappa_{d}(\varepsilon)\]
uniformly in \(x \in M\), with constants uniform over
\(0 < \rho \leq \rho_{0}\) and
\(0 < \varepsilon < \varepsilon_{0}\rho\). The boundary case is the
essential point: near \(\partial M\), after flattening, the local model
is a half-ball with Neumann condition on the flat face, and even
reflection across the Neumann boundary converts the half-space condenser
into a full-space condenser with comparable energy (Lemma A.2). Thus the
capacity scale near the reflecting boundary is the same as in the
interior; only the constant changes.
\begin{proof}[Proof of Proposition 5.2] \emph{Upper bound.} Let \(v\) be a
local condenser potential: \(v \geq 1\) on \(B_{M}(x,\varepsilon)\),
supported in \(B_{M}(x,\rho)\),
\(\mathcal{E}(v,v) \leq C\kappa_{d}(\varepsilon)\). Set
\(\bar{v} = \int_{M}^{}v\, d\mu\) and \(w = v - \bar{v}\); subtracting a
constant preserves the energy and enforces \(\int_{M}^{}w\, d\mu = 0\).
Since \(v\) is supported in a ball of fixed small radius and bounded
after truncation, \(\bar{v} \leq 1 - c\) for a constant \(c > 0\)
depending only on the local geometry, so \(w \geq c\) on
\(B_{M}(x,\varepsilon)\) and \(c^{- 1}w\) is admissible. Hence
\(\operatorname{Cap}_{N}\left( B_{M}(x,\varepsilon) \right) \leq c^{- 2}\,\mathcal{E}(v,v) \leq C\kappa_{d}(\varepsilon)\).
\emph{Lower bound.} Let \(u\) be admissible for
\(\operatorname{Cap}_{N}\left( B_{M}(x,\varepsilon) \right)\): \(u \geq 1\)
quasi-everywhere on \(B_{M}(x,\varepsilon)\) and
\(\int_{M}^{}u\, d\mu = 0\). Fix \(\rho \leq \rho_{0}\) below the local
chart scale, write \(W = B_{M}(x,\rho)\backslash B_{M}(x,\rho/2)\) for
the reference annulus, and let
\[\alpha = \frac{1}{\mu(W)}\int_{W}^{}u\, d\mu\]
be the average of \(u\) over \(W\). We distinguish two cases according
to the size of \(\alpha\).
\emph{Case 1:} \(\alpha \leq 1/2\) \emph{(the potential drops inside the
patch).} We construct a local condenser competitor from \(u\). Let
\(\chi\) be a Lipschitz cutoff with \(\chi = 1\) on \(B_{M}(x,\rho/2)\),
\(\chi = 0\) on \(M\backslash B_{M}(x,\rho)\), and
\(|\nabla\chi| \leq C/\rho\), and define
\[v = \chi \cdot \min\left( 1,\mspace{6mu} 2(u - \alpha)_{+} \right).\]
On \(B_{M}(x,\varepsilon)\) (recall
\(\varepsilon < \varepsilon_{0}\rho \leq \rho/2\)) we have \(u \geq 1\),
hence \(2(u - \alpha) \geq 2\left( 1 - \frac{1}{2} \right) = 1\) and
\(\chi = 1\), so \(v = 1\) there; and \(v = 0\) on
\(\partial B_{M}(x,\rho) \cap M\). Thus \(v\) is admissible for the
local condenser. Its energy is estimated by the product rule:
\begin{align*}
\mathcal{E}(v,v) &\leq 2\int_{B_{M}(x,\rho)}^{}\chi^{2}\,\left| \nabla \min\left( 1,2(u - \alpha)_{+} \right) \right|^{2}d\mu\\
&\quad + 2\int_{W}^{}\min\left( 1,2(u - \alpha)_{+} \right)^{2}\,|\nabla\chi|^{2}\, d\mu,
\end{align*}
where the second integral is restricted to \(W\) because \(\nabla\chi\)
is supported there. The first term is at most \(8\,\mathcal{E}(u,u)\),
since truncations are normal contractions. For the second term,
\(|\nabla\chi|^{2} \leq C\rho^{- 2}\) and
\(\min\left( 1,2(u - \alpha)_{+} \right)^{2} \leq 4(u - \alpha)^{2}\),
so it is bounded by
\[C\rho^{- 2}\int_{W}^{}(u - \alpha)^{2}\, d\mu \leq C\rho^{- 2} \cdot C_{ann}\,\rho^{2}\int_{W}^{}|\nabla u|^{2}\, d\mu \leq C'\,\mathcal{E}(u,u),\]
where the middle step is the Poincaré inequality on the annulus \(W\)
with its own mean \(\alpha\); its constant scales as \(\rho^{2}\) and is
uniform in \(x\) because, in the charts of Appendix A.3, \(W\) is
bi-Lipschitz to a Euclidean annulus (or half-annulus, for \(x\) near
\(\partial M\)) with uniform constants. Altogether
\(\mathcal{E}(v,v) \leq C\,\mathcal{E}(u,u)\), and the local condenser
lower bound of Proposition A.3 gives
\[c_{loc}\,\kappa_{d}(\varepsilon) \leq \operatorname{Cap}_{loc}(x,\varepsilon;\rho)\mathcal{\leq E}(v,v) \leq C\,\mathcal{E}(u,u),\]
hence \(\mathcal{E}(u,u) \geq c\,\kappa_{d}(\varepsilon)\) in this case.
\emph{Case 2:} \(\alpha > 1/2\) \emph{(the potential stays high in the
patch).} Then \(\int_{W}^{}u\, d\mu > \mu(W)/2\), and the mean-zero
constraint forces compensation:
\(\int_{M\backslash W}^{}u\, d\mu < - \mu(W)/2\). By the Cauchy-Schwarz
inequality on each region,
\[\parallel u \parallel_{L^{2}(\mu)}^{2} \geq \frac{\left( \int_{W}^{}u\, d\mu \right)^{2}}{\mu(W)} + \frac{\left( \int_{M\backslash W}^{}u\, d\mu \right)^{2}}{\mu(M\backslash W)} \geq \frac{\mu(W)}{4} \geq c(\rho) > 0,\]
where \(c(\rho)\) is uniform in \(x\) by the volume regularity of
Assumption 2.3 applied at the fixed scale \(\rho\). Since \(u\) has mean
zero, the global Poincaré inequality gives
\[\mathcal{E}(u,u) \geq \lambda_{1}(M)\, \parallel u \parallel_{L^{2}(\mu)}^{2} \geq \lambda_{1}(M)\, c(\rho) = :c_{0} > 0,\]
a constant independent of \(\varepsilon\). As
\(\kappa_{d}(\varepsilon) \rightarrow 0\) when
\(\varepsilon \downarrow 0\), we have
\(c_{0} \geq c\,\kappa_{d}(\varepsilon)\) for all
\(\varepsilon < \varepsilon_{\lg}\) after decreasing
\(\varepsilon_{\lg}\) if necessary.
In both cases \(\mathcal{E}(u,u) \geq c\,\kappa_{d}(\varepsilon)\);
taking the infimum over admissible \(u\) proves the lower bound.
Uniformity in \(x\) holds because the chart, volume, flattening,
annulus-Poincaré, and global Poincaré constants are uniform over the
compact admissible domain. \end{proof}
\subsection{The uniform small-hole eigenvalue theorem}
\begin{theorem}[Uniform small-hole eigenvalue estimate]
There exist constants \(c_{sh},C_{sh} > 0\) and
\(\varepsilon_{sh} > 0\), depending only on the geometric and
analytic data of \(M\), such that for every \(x \in M\)
and every \(0 < \varepsilon < \varepsilon_{sh}\),

\[c_{sh}\,\kappa_{d}(\varepsilon) \leq \lambda_{B_{M}(x,\varepsilon)} \leq C_{sh}\,\kappa_{d}(\varepsilon).\]
\end{theorem}
\begin{proof} By Proposition 5.2,
\(\operatorname{Cap}_{N}\left( B_{M}(x,\varepsilon) \right) \asymp \kappa_{d}(\varepsilon)\)
uniformly in \(x\) for \(\varepsilon < \varepsilon_{\lg}\). Since
\(\kappa_{d}(\varepsilon) \rightarrow 0\) as
\(\varepsilon \downarrow 0\), choose
\(\varepsilon_{sh} \leq \varepsilon_{\lg}\) small enough that
\(C_{\lg}\,\kappa_{d}(\varepsilon) \leq \eta_{ce}\) for all
\(\varepsilon < \varepsilon_{sh}\), so that
\(\operatorname{Cap}_{N}\left( B_{M}(x,\varepsilon) \right) \leq \eta_{ce}\) for
every \(x \in M\); Proposition 5.1 then gives
\(\lambda_{B_{M}(x,\varepsilon)} \asymp \operatorname{Cap}_{N}\left( B_{M}(x,\varepsilon) \right)\),
with \(c_{sh} = c_{ce}\, c_{\lg}\) and \(C_{sh} = C_{ce}\, C_{\lg}\),
and the two comparabilities combine. \end{proof}
Explicitly:
\(\lambda_{B_{M}(x,\varepsilon)} \asymp \left( \log\varepsilon^{- 1} \right)^{- 1}\)
in dimension \(2\) and \(\asymp \varepsilon^{d - 2}\) in dimensions
\(d \geq 3\). The uniformity over boundary targets is essential: if
boundary targets had a smaller capacity scale, the recovery time would
be controlled by the hardest boundary regions, and the universal theorem
would degrade.
\begin{corollary}[Uniform hitting time for ball families]
There exist constants \(C > 0\) and
\(\varepsilon_{0} > 0\) such that for every family
\(\mathcal{A = \{}B_{M}\left( x_{1},\varepsilon \right),\ldots,B_{M}\left( x_{m},\varepsilon \right)\}\)
with \(0 < \varepsilon < \varepsilon_{0}\), every starting
point \(x \in M\), every \(\delta \in (0,1)\), and every

\[T \geq C\left( s_{0} + \kappa_{d}(\varepsilon)^{- 1}\left( \log m + \log\frac{1}{\delta} \right) \right),\]

one has
\(\mathbb{P}_{x}\left( H_{T}\left( \mathcal{A} \right) \right) \geq 1 - \delta\).
\end{corollary}
\begin{proof} Theorem 5.3 gives
\(\lambda_{B_{M}\left( x_{i},\varepsilon \right)} \geq c_{sh}\,\kappa_{d}(\varepsilon)\)
for every \(i\); apply Corollary 4.5 with
\(\lambda_{\min} = c_{sh}\,\kappa_{d}(\varepsilon)\) and absorb
\(c_{sh}\) and \(a_{hk}\) into \(C\). \end{proof}
The sharp, dimension-dependent leading constants for the small-hole
eigenvalue --- the planar and higher-dimensional asymptotics together
with the boundary-layer profile --- are not needed for the recovery and
optimality theorems of Sections 6 and 7; they are developed separately
in Appendix D.
\section{Main upper bounds for homological recovery}
We now combine the deterministic bridge of Section 3 with the spectral
estimates of Sections 4 and 5. The general statement is
feature-adaptive; universal recovery is the extremal special case in
which the certificate is a full intrinsic net. The probabilistic method
is identical in the two regimes --- what changes is the deterministic
meaning of the target family.
\subsection{Feature-adaptive detection and recovery}
\begin{theorem}[Feature-adaptive detection and recovery from
spectral certificates]
Let
\(\mathcal{A = \{}A_{1},\ldots,A_{m}\}\) be a degree-\(q\),
radius-\(r\) homological certificate for \(M\). Then, for
every starting point \(z \in M\) and every
\(T \geq s_{0}\),

\[\mathbb{P}_{z}\left( R_{q}(T,r) \right) \geq 1 - a_{hk}\,\Phi_{\mathcal{A}}\left( T - s_{0} \right).\]

In particular,
\(\mathbb{P}_{z}\left( H_{q}\left( \Gamma_{T}^{r} \right) \cong H_{q}(M) \right) \geq 1 - \delta\)
whenever
\(a_{hk}\,\Phi_{\mathcal{A}}\left( T - s_{0} \right) \leq \delta\).
If \(\mathcal{A}\) is a full homological certificate at radius
\(r\), the same bound holds for the simultaneous event
\(R(T,r)\).
\end{theorem}
\begin{proof} Proposition 3.4(i) gives
\(H_{T}\left( \mathcal{A} \right) \subseteq R_{q}(T,r)\), hence
\(R_{q}(T,r)^{c} \subseteq H_{T}\left( \mathcal{A} \right)^{c}\), and
Theorem 4.4 bounds
\(\mathbb{P}_{z}\left( H_{T}\left( \mathcal{A} \right)^{c} \right) \leq a_{hk}\,\Phi_{\mathcal{A}}\left( T - s_{0} \right)\).
The full-certificate statement is identical with \(R_{q}(T,r)\) replaced
by \(R(T,r)\). \end{proof}
If all targets satisfy \(\lambda_{A} \geq \lambda_{\min}\), then
\(\Phi_{\mathcal{A}}(u) \leq m\, e^{- \lambda_{\min}u}\) and Theorem 6.1
gives recovery once
\(T \geq s_{0} + \lambda_{\min}^{- 1}\left( \log m + \log\left( a_{hk}/\delta \right) \right)\).
The most useful instance is ball-like certificates, where Theorem 5.3
identifies \(\lambda_{\min}\).
\begin{remark}[what is, and is not, feature-adaptive]
Theorem 6.1 is
conditional: its conclusion is the two-sided event \(R_{q}(T,r)\), and
it delivers the feature-adaptive rate only for a family \(A\) that is a
certificate in the sense of Definition 3.3. The certificates we exhibit
of size \(m \asymp N_{\varepsilon}(M)\) are full nets (Proposition
3.4(ii)), for which the conclusion is the universal reconstruction
theorem of Section 6.2; in the extreme degrees the certificate property
is degenerate, since \(H_{0}\) is fixed by connectedness of the trace
and \(H_{d}(M) = 0\) for \(M \subseteq \mathbb{R}^{d}\). For a family
strictly smaller than a net in an intermediate degree \(0 < q < d\),
what the same spectral hitting bound establishes is detection ---
surjectivity of
\(\iota_{*}:H_{q}\left( \Gamma_{T}^{r} \right) \rightarrow H_{q}(M)\),
in the one-sided sense of the Remark after Proposition 3.4 --- and not
the two-sided isomorphism. Thus the saving of \(\log m\) for
\(\log N_{\varepsilon}(M)\) is unconditional for detection, whereas the
corresponding feature-adaptive reconstruction statement is conditional
on the existence of a small certificate, which we leave open. The
probabilistic estimate is blind to the distinction: it bounds the time
to visit a finite family, and Theorem 6.1 reads that bound as
reconstruction precisely when \(A\) is a certificate, and as detection
otherwise.
\end{remark}
\begin{corollary}[Ball-like certificates]
There exist
constants \(C > 0\) and \(\varepsilon_{0} > 0\), depending
only on the geometric and analytic data of \(M\), such that the
following holds. Let \(0 < \varepsilon < \varepsilon_{0}\) and
let
\(\mathcal{A = \{}B_{M}\left( y_{1},\varepsilon \right),\ldots,B_{M}\left( y_{m},\varepsilon \right)\}\)
be a degree-\(q\), radius-\(r\) homological
certificate for \(M\). Then, for every \(z \in M\) and
\(\delta \in (0,1)\),

\begin{multline*}
\mathbb{P}_{z}\left( H_{q}\left( \Gamma_{T}^{r} \right) \cong H_{q}(M) \right) \geq 1 - \delta\\
\text{whenever}\quad T \geq C\left( s_{0} + \kappa_{d}(\varepsilon)^{- 1}\left( \log m + \log\tfrac{1}{\delta} \right) \right).
\end{multline*}

If \(\mathcal{A}\) is a full certificate, the same
condition gives simultaneous recovery of all homology groups.
\end{corollary}
\begin{proof} Theorem 5.3 gives
\(\lambda_{A} \geq c_{sh}\,\kappa_{d}(\varepsilon)\) for every
\(A\mathcal{\in A}\); the resulting hitting bound is Corollary 5.4
(equivalently,~Corollary 4.5 with
\(\lambda_{\min} = c_{sh}\,\kappa_{d}(\varepsilon)\)), which the
certificate property of Proposition 3.4(i) --- through Theorem 6.1 ---
reads as recovery after absorbing constants. \end{proof}
\subsection{Universal recovery}
\begin{theorem}[Universal homological recovery]
There exist
constants \(C,C_{low},C_{up},\varepsilon_{0} > 0\), depending
only on the geometric and analytic data of \(M\), such that for
every \(0 < \varepsilon < \varepsilon_{0}\), every
\(\delta \in (0,1)\), every starting point \(z \in M\), and
every

\[T \geq C\left( s_{0} + \kappa_{d}(\varepsilon)^{- 1}\left( \log N_{\varepsilon}(M) + \log\frac{1}{\delta} \right) \right),\]

one has, for every reconstruction radius \(r\) with
\(C_{low}\,\varepsilon \leq r \leq C_{up}\, \operatorname{reach}(M)\),

\[\mathbb{P}_{z}\left( H_{q}\left( \Gamma_{T}^{r} \right) \cong H_{q}(M)\ \text{for every }q = 0,\ldots,d \right) \geq 1 - \delta.\]
\end{theorem}
\begin{proof} By Proposition 3.4(ii), the net family
\(\mathcal{A}_{\varepsilon} = \{ B_{M}\left( x_{i},\varepsilon \right):i = 1,\ldots,N_{\varepsilon}(M)\}\)
is a full homological certificate of size \(m = N_{\varepsilon}(M)\) at
every radius in the window. Apply Corollary 6.2. \end{proof}
Note in particular that on the event
\(H_{T}\left( \mathcal{A}_{\varepsilon} \right)\) the trace is
intrinsically \(2\varepsilon\)-dense in \(M\) (proof of Proposition
3.4), so the same time bound also gives, after adjusting constants,
\[\mathbb{P}_{z}\left( \Gamma_{T}\text{ is intrinsically }\varepsilon\text{-dense in }M \right) \geq 1 - \delta,\]
a density statement used again in Section 8.
\begin{corollary}[Dimensional forms]
Under the volume
regularity assumption,
\(\log N_{\varepsilon}(M) \asymp d\log\varepsilon^{- 1}\), and for
fixed confidence level \(\delta\) the universal recovery scale of
Theorem 6.3 is

\[T \asymp s_{0} + \left( \log\varepsilon^{- 1} \right)^{2}\quad(d = 2),\quad\quad T \asymp s_{0} + \varepsilon^{2 - d}\log\varepsilon^{- 1}\quad(d \geq 3),\]

while the feature-adaptive scale of Corollary 6.2 with
\(m = O(1)\) witness regions is

\[T \asymp s_{0} + \log\varepsilon^{- 1}\quad(d = 2),\quad\quad T \asymp s_{0} + \varepsilon^{2 - d}\quad(d \geq 3).\]
\end{corollary}
The comparison exhibits the structural point of the paper. The universal
bound has three factors with distinct meanings: \(s_{0}\) is global
spectral access (burn-in, typically \(\asymp \lambda_{1}(M)^{- 1}\));
\(\kappa_{d}(\varepsilon)^{- 1}\) is local capacitary detection of one
small target; and \(\log N_{\varepsilon}(M)\) is the coupon-collector
price of full metric coverage. Feature-adaptive recovery replaces the
covering logarithm by \(\log m\): for finitely many robust features the
coupon-collector cost over the whole domain disappears --- a full
logarithmic factor in dimension \(2\), the factor
\(\log\varepsilon^{- 1}\) in dimensions \(d \geq 3\) --- while at the
opposite extreme, if the topology to be recovered varies at the covering
scale, any certificate may require \(m \asymp N_{\varepsilon}(M)\) and
the feature-adaptive theorem collapses back to the universal one. The
certificate framework also distinguishes degree-wise from simultaneous
recovery: a degree-\(q\) certificate recovers \(H_{q}\) alone, which may
be substantially faster than full recovery when only one feature is of
interest.
\section{Lower bounds and optimality}
We now show that the terms of the upper bounds are not artifacts of the
proof. Three independent obstructions can prevent recovery:
\emph{spectral access} (a bottleneck forces a time of order
\(\lambda_{1}(M)^{- 1}\) before the relevant region is reached),
\emph{capacitary detection} (a feature hidden in an \(\varepsilon\)-ball
cannot be detected before the ball is hit, a time of order
\(\kappa_{d}(\varepsilon)^{- 1}\)), and \emph{feature complexity}
(resolving \(m\) independent hidden features is a coupon-collector
problem at the capacitary scale, forcing
\(\kappa_{d}(\varepsilon)^{- 1}\log m\)). The lower bounds are minimax
statements, using Le Cam's two-point method {[}22{]} and its
multi-hypothesis extensions {[}23{]}; the geometric content lies in
constructing domains whose topology differs only in regions the
trajectory has not yet visited.
\subsection{Observation laws and the indistinguishability principle}
Let \(\mathcal{M}\) be a class of admissible domains satisfying the
assumptions of Sections 2-3 with uniform constants. For
\(M\mathcal{\in M}\), let \(P_{M}^{T}\) denote the law of the observed
trajectory \(\left( X_{t} \right)_{0 \leq t \leq T}\) started from a
prescribed point \(x_{0}\), chosen in a region on which the compared
domains coincide. An estimator of degree-\(q\) homology is a measurable
map
\({\widehat{H}}_{q} = {\widehat{H}}_{q}\left( \left( X_{t} \right)_{0 \leq t \leq T} \right)\)
with values in a finite set of labels encoding the alternatives, and the
minimax error over a finite subfamily
\(\mathcal{M}_{0}\mathcal{\subset M}\) is
\[\mathcal{R}_{T}\left( \mathcal{M}_{0} \right) = \inf_{{\widehat{H}}_{q}}\max_{M \in \mathcal{M}_{0}}P_{M}^{T}\left( {\widehat{H}}_{q} \neq H_{q}(M) \right).\]
All constructions use one principle: if two domains coincide outside a
region \(A\) and the trajectory has not entered \(A\), the observations
are identical.
The matching with the upper bounds is read per domain. For each fixed
admissible \(M\) --- with its own spectral gap \(\lambda_{1}(M)\) and
reach \(\operatorname{reach}(M)\) --- the bounds below show that the observation time
cannot be reduced below the stated scale, and are compared with the
upper bounds for that same \(M\). In the dumbbell family the index
\(\eta\) labels fixed admissible domains \(M_{\eta}\); although
\(\operatorname{reach}\left( M_{\eta} \right) \rightarrow 0\) as the neck closes, so
that the reconstruction window
\(C_{low}\,\varepsilon \leq r \leq C_{up}\operatorname{reach}\left( M_{\eta} \right)\)
of Theorem 6.3 shrinks with \(\eta\) for that domain, the hidden
\(\varepsilon\)-features are confined to a fixed chamber \(R\) whose
reach and chart constants depend only on \(R\) and not on \(\eta\). The
feature-detection lower bound is therefore uniform in \(\eta\), and the
comparison is between the upper and lower bounds at each fixed
\(M_{\eta}\), not over a single class with \(\eta\)-uniform
reconstruction constants.
\begin{lemma}[Indistinguishability before hitting]
Let
\(M_{0},M_{1}\) be admissible domains with
\(H_{q}\left( M_{0} \right) \not\cong H_{q}\left( M_{1} \right)\). Suppose
their reflected Brownian motions, started from a common point
\(x_{0}\), can be coupled so that the observed trajectories agree
up to time \(\tau_{A}\), where \(A\) contains all
geometric differences between \(M_{0}\) and \(M_{1}\).
Then every estimator satisfies

\[\max_{\theta \in \{ 0,1\}}P_{M_{\theta}}^{T}\left( {\widehat{H}}_{q} \neq H_{q}\left( M_{\theta} \right) \right) \geq \frac{1}{2}\inf_{\theta \in \{ 0,1\}}P_{M_{\theta}}\left( \tau_{A} > T \right).\]
\end{lemma}
\begin{proof} The coupling implies
\(\parallel P_{M_{0}}^{T} - P_{M_{1}}^{T} \parallel_{TV} \leq 1 - \inf_{\theta}P_{M_{\theta}}\left( \tau_{A} > T \right)\),
since on \(\{\tau_{A} > T\}\) the observed paths coincide. Le Cam's
inequality {[}22{]},
\[\inf_{\widehat{\theta}}\max_{\theta \in \{ 0,1\}}P_{M_{\theta}}^{T}\left( \widehat{\theta} \neq \theta \right) \geq \frac{1}{2}\left( 1 - \parallel P_{M_{0}}^{T} - P_{M_{1}}^{T} \parallel_{TV} \right),\]
applied to the test induced by \({\widehat{H}}_{q}\) (correct for one
alternative forces incorrect for the other, since the homologies
differ), gives the claim. \end{proof}
\begin{lemma}[Survival from a killed eigenfunction]
Let
\(B \subseteq M\) be closed of positive capacity, let
\(\lambda_{B}\) denote the principal eigenvalue of the process
killed on \(B\), and let \(\varphi_{B} \geq 0\) be a
corresponding principal eigenfunction normalised by
\(\parallel \varphi_{B} \parallel_{L^{\infty}(M)} = 1\). Then for
every \(x \in M\) and every \(t \geq 0\),

\[\mathbb{P}_{x}\left( \tau_{B} > t \right) \geq \varphi_{B}(x)\, e^{- \lambda_{B}t}.\]
\end{lemma}
\begin{proof} The killed semigroup
\(\left( P_{t}^{B}f \right)(x) = \mathbb{E}_{x}\left\lbrack f\left( X_{t} \right);\tau_{B} > t \right\rbrack\)
is positivity-preserving and monotone: \(0 \leq f \leq g\) implies
\(0 \leq P_{t}^{B}f \leq P_{t}^{B}g\). Since
\(0 \leq \varphi_{B} \leq 1 = \parallel \varphi_{B} \parallel_{L^{\infty}}\),
the constant function \(1\) dominates \(\varphi_{B}\), so
\(P_{t}^{B}1 \geq P_{t}^{B}\varphi_{B} = e^{- \lambda_{B}t}\varphi_{B}\),
the identity holding because \(\varphi_{B}\), extended by \(0\) on
\(B\), is a principal eigenfunction of the killed generator. Evaluating
at \(x\) gives
\(\mathbb{P}_{x}\left( \tau_{B} > t \right) = \left( P_{t}^{B}1 \right)(x) \geq e^{- \lambda_{B}t}\varphi_{B}(x)\).
\end{proof}
In the two constructions below the principal eigenfunction is bounded
below at the starting point. For a small target
\(A_{\varepsilon} = B_{M}(x,\varepsilon)\) the normalised eigenfunction
\(\varphi_{A_{\varepsilon}}\) tends to the constant \(1\) locally
uniformly away from the shrinking hole, by the same small-obstacle
perturbation theory that yields Theorem 5.3; hence there exist
\(\rho_{0},c_{\varphi} > 0\) with
\(\varphi_{A_{\varepsilon}}\left( x_{0} \right) \geq c_{\varphi}\)
whenever \(\operatorname{dist}\left( x_{0},A_{\varepsilon} \right) \geq \rho_{0}\) and
\(\varepsilon\) is small. For a chamber \(R\) in a dumbbell the
eigenfunction \(\varphi_{R}\) of the process killed on \(R\) is
comparable to its maximum on the bulk \(L_{0}\) of the opposite chamber
by the Harnack inequality, so
\(\varphi_{R}\left( x_{0} \right) \geq c_{\varphi}\) for
\(x_{0} \in L_{0}\); moreover \(\lambda_{R} \lesssim \lambda_{1}(M)\)
for the family, by Lemma C.1 and the thin-neck Neumann eigenvalue
asymptotics; only this direction is used below.
\begin{remark}[hidden features of arbitrary degree]
The constructions
of this section describe the hidden feature as a planar obstacle
creating one \(H_{1}\)-class; this is the case \(q = d - 1 = 1\) of a
modification available in every degree \(1 \leq q \leq d - 1\) and
dimension \(d \geq 2\), the full range recovered in Section 6. Fix
\(x \in M\) and \(\varepsilon > 0\) small. For \(q = d - 1\) take
\(M_{1} = M_{0}\backslash B(x,\varepsilon/2)\), removing an
\(\varepsilon\)-scale ball to open a spherical void and adjoin one
generator to \(H_{d - 1}\). For \(1 \leq q \leq d - 2\) (hence
\(d \geq 3\)) let \(\Sigma \subset B(x,\varepsilon/2)\) be a smoothly
embedded \((d - q - 1)\)-sphere of radius \(\varepsilon/4\) and take
\(M_{1} = M_{0}\backslash N(\Sigma)\), with \(N(\Sigma)\) the solid
tubular neighbourhood of \(\Sigma\) of radius \(\varepsilon/8\); by
Alexander duality its complement carries exactly one additional
\(H_{q}\)-generator. In either case the removed obstacle \(O\) lies in
\(B(x,\varepsilon)\), contains a ball of radius \(\asymp \varepsilon\),
and is bounded by a smooth \(\varepsilon\)-scale hypersurface, so
\(M_{1}\) is an admissible reflecting domain coinciding with \(M_{0}\)
outside \(B(x,\varepsilon)\) with
\(H_{q}\left( M_{1} \right) \not\cong H_{q}\left( M_{0} \right)\), and by
monotonicity of capacity together with Proposition 5.2 its Neumann
capacity satisfies \(\operatorname{Cap}_{N}(O) \asymp \kappa_{d}(\varepsilon)\). The
hitting analysis is then identical to the planar case: the two domains
are indistinguishable until the trajectory enters \(B(x,\varepsilon)\),
an event of capacitary rate \(\kappa_{d}(\varepsilon)^{- 1}\) by Theorem
5.3, so Lemmas 7.1 and 7.2 and the constructions of Propositions 7.3,
7.4 and 7.6 and Theorem 7.8 carry over verbatim in every such degree,
each hidden obstacle contributing one independent \(H_{q}\)-class. The
extreme degrees remain degenerate, as recorded after Proposition 3.4:
\(H_{0}\) is fixed by the connectedness of the trace and
\(H_{d}(M) = 0\) for \(M \subseteq \mathbb{R}^{d}\).
\end{remark}
\subsection{Bottleneck lower bound: necessity of spectral access}
\begin{figure}[htbp]
\centering
\robustfig[0.94\linewidth]{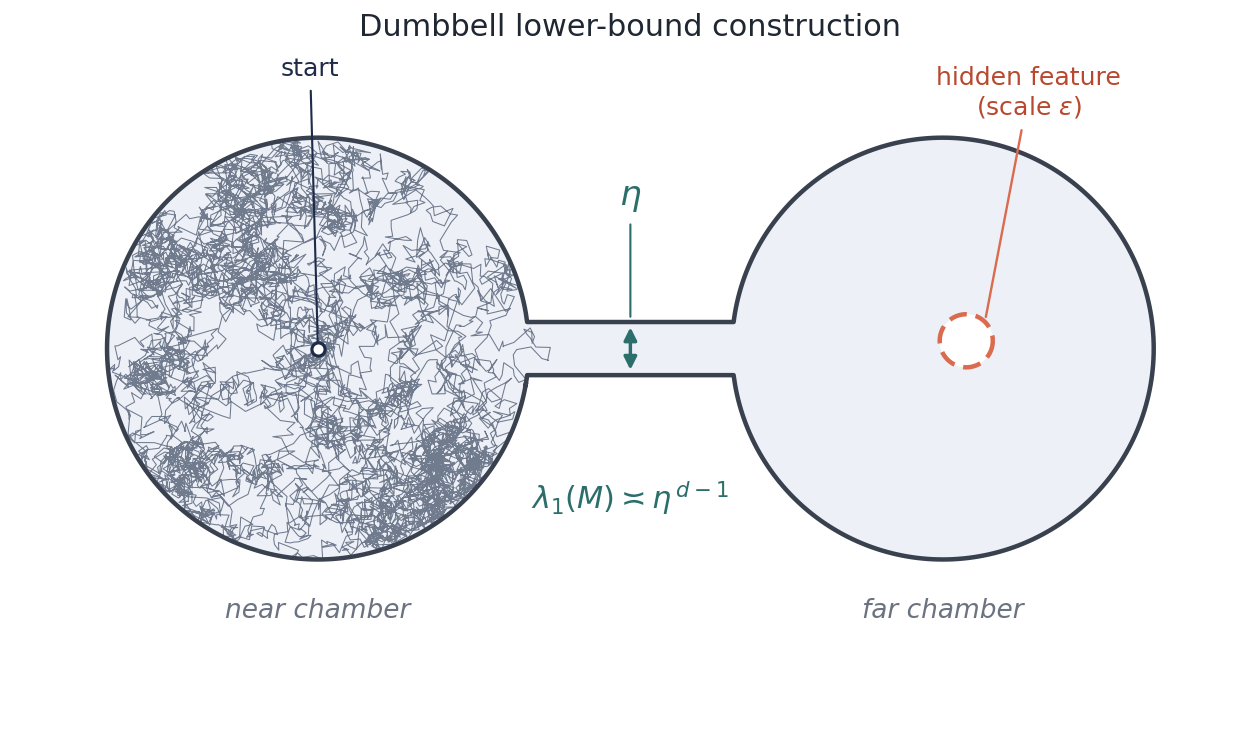}
\caption{The dumbbell lower-bound construction. Two chambers are joined by a neck of width \(\eta\), and a topological feature of scale \(\varepsilon\) is hidden in the far chamber. The Neumann spectral gap satisfies \(\lambda_{1}(M) \asymp \eta^{\, d - 1}\), so a reflected Brownian motion started in the near chamber (trace shown) remains trapped there for a time of order \(\lambda_{1}(M)^{- 1} \asymp \eta^{- (d - 1)}\) before it can cross the neck. Detecting the far feature therefore costs this spectral-access time \emph{in addition to} the capacity time \(\kappa_{d}(\varepsilon)^{- 1}\), which is how the construction forces the spectral-access and capacitary obstructions to be simultaneously necessary (Theorem 7.8).}
\label{fig:3}
\end{figure}
Consider a dumbbell domain \(M = L \cup N \cup R\), two large chambers
joined by a narrow neck, with the process started in the left chamber
\(L\) and the distinguishing feature hidden in a subregion
\(A \subset R\): for instance, in dimension \(2\), \(M_{1}\) contains a
small obstacle in \(R\) creating an additional \(H_{1}\)-class while
\(M_{0}\) does not. For a family of dumbbells with uniformly regular
chambers and increasingly narrow necks, Lemma 7.2 gives, with
\(\tau_{R} = \inf\{ t \geq 0:X_{t} \in R\}\),
\[\mathbb{P}_{x_{0}}\left( \tau_{R} > T \right) \geq c\quad\quad\text{whenever}\quad\quad T \leq c'\,\lambda_{1}(M)^{- 1},\]
with \(c,c' > 0\) independent of the neck width: the first nonzero
Neumann eigenvalue measures the rate of communication between the
chambers.
\begin{proposition}[Bottleneck lower bound]
There exists a
family of admissible dumbbell domains
\(\{\left( M_{0}^{\eta},M_{1}^{\eta} \right)\}_{\eta > 0}\) with
common starting point \(x_{0}\) in the left chamber and
\(H_{q}\left( M_{0}^{\eta} \right) \not\cong H_{q}\left( M_{1}^{\eta} \right)\)
for some degree \(q\), such that for all sufficiently
small \(\eta\),

\[\inf_{{\widehat{H}}_{q}}\max_{\theta \in \{ 0,1\}}P_{M_{\theta}^{\eta}}^{T}\left( {\widehat{H}}_{q} \neq H_{q}\left( M_{\theta}^{\eta} \right) \right) \geq c\quad\quad\text{whenever}\quad\quad T \leq c'\,\lambda_{1}\left( M_{\theta}^{\eta} \right)^{- 1},\]

with \(c,c' > 0\) independent of \(\eta\).
\end{proposition}
\begin{proof} We realize the family explicitly. Let \(L\) and \(R\) be
fixed bounded smooth chambers and, for \(\eta > 0\), let \(N^{\eta}\) be
a neck of fixed length \(\ell\) joining them whose cross-section
has \((d - 1)\)-dimensional measure of order \(\eta^{d - 1}\), so that
\(M^{\eta} = L \cup N^{\eta} \cup R\) is an admissible reflecting
domain. The distinguishing feature is placed in a region \(A \subset R\)
whose modification changes degree-\(q\) homology --- in \(d = 2\), a
small obstacle in \(R\) creating one additional \(H_{1}\)-class, present
in \(M_{1}^{\eta}\) and absent in \(M_{0}^{\eta}\); the two domains
coincide on \(L \cup N^{\eta} \cup (R\backslash A)\). The common
starting point \(x_{0}\) lies in a fixed interior bulk
\(L_{0} \subset \subset L\), at fixed positive distance from
\(N^{\eta}\).
The two domains differ only inside \(A \subset R\), and their reflected
motions, killed on \(R\), agree in law until \(\tau_{R}\); since
\(A \subset R\) one has \(\tau_{A} \geq \tau_{R}\), so by Lemma 7.1 the
minimax error is at least
\(\frac{1}{2}\, P_{x_{0}}\left( \tau_{R} > T \right)\). By Lemma C.1 the
killed eigenvalue obeys \(\lambda_{R} \leq C_{1}\eta^{d - 1}\), and by
Lemma C.3 the survival probability satisfies
\(P_{x_{0}}\left( \tau_{R} > T \right) \geq c_{\varphi}\, e^{- \lambda_{R}T}\)
with \(c_{\varphi} > 0\) independent of \(\eta\). Finally, by the
classical thin-neck Neumann eigenvalue asymptotics for dumbbell domains
{[}26, 27{]}, \(\lambda_{1}\left( M^{\eta} \right) \asymp \eta^{d - 1}\)
as \(\eta \rightarrow 0\); in particular
\(\lambda_{R} \leq C_{1}\eta^{d - 1} \leq C_{2}\,\lambda_{1}\left( M^{\eta} \right)\).
Hence for \(T \leq c'\lambda_{1}\left( M^{\eta} \right)^{- 1}\) we have
\(\lambda_{R}T \leq C_{2}c'\), so
\(P_{x_{0}}\left( \tau_{R} > T \right) \geq c_{\varphi}e^{- C_{2}c'}\),
and the minimax error is at least
\(\frac{1}{2}c_{\varphi}e^{- C_{2}c'} = :c > 0\). The conductance bound
of Lemma C.1 and the eigenfunction lower bound of Lemma C.3, proved in
Appendix C, make the preceding survival estimate quantitative.
\end{proof}
\subsection{One-feature lower bound: necessity of inverse capacity}
\begin{proposition}[One-feature capacitary lower bound]
There exist constants \(c,c' > 0\) and families of pairs of
admissible domains
\(\left( M_{0}^{\varepsilon},M_{1}^{\varepsilon} \right)\),
coinciding outside a ball-like region
\(A_{\varepsilon} = B_{M}(x,\varepsilon)\) inside which the
modification changes degree-\(q\) homology, such that for all
sufficiently small \(\varepsilon\),

\[\inf_{{\widehat{H}}_{q}}\max_{\theta \in \{ 0,1\}}P_{M_{\theta}^{\varepsilon}}^{T}\left( {\widehat{H}}_{q} \neq H_{q}\left( M_{\theta}^{\varepsilon} \right) \right) \geq c\quad\quad\text{whenever}\quad\quad T \leq c'\,\kappa_{d}(\varepsilon)^{- 1}.\]
\end{proposition}
\begin{proof} Fix \(x_{*} \in M\) in a fixed admissible ambient domain
and put \(A_{\varepsilon} = B_{M}\left( x_{*},\varepsilon \right)\). The
two domains \(M_{0}^{\varepsilon},M_{1}^{\varepsilon}\) coincide on
\(M\backslash A_{\varepsilon}\), the modification changing degree-\(q\)
homology being supported in \(A_{\varepsilon}\); the common starting
point \(x_{0}\) is at fixed intrinsic distance \(\geq \rho_{0}\) from
\(A_{\varepsilon}\). Because the two domains agree on
\(M\backslash A_{\varepsilon}\) and the motion is killed upon entering
\(A_{\varepsilon}\), the killed process has the same law under both
models, and so does the survival probability
\(P\left( \tau_{A_{\varepsilon}} > T \right)\). By Lemma 7.1 the minimax
error is at least
\(\frac{1}{2}\, P_{x_{0}}\left( \tau_{A_{\varepsilon}} > T \right)\).
By Theorem 5.3,
\(\lambda_{A_{\varepsilon}} \leq C_{sh}\,\kappa_{d}(\varepsilon)\), and
by Lemma C.2 the survival probability obeys
\(P_{x_{0}}\left( \tau_{A_{\varepsilon}} > T \right) \geq c_{\varphi}\, e^{- \lambda_{A_{\varepsilon}}T}\),
with \(c_{\varphi} > 0\) independent of \(\varepsilon\) and of the
centre \(x_{*}\). Hence for \(T \leq c'\,\kappa_{d}(\varepsilon)^{- 1}\)
we obtain \(\lambda_{A_{\varepsilon}}T \leq C_{sh}c'\) and
\(P_{x_{0}}\left( \tau_{A_{\varepsilon}} > T \right) \geq c_{\varphi}e^{- C_{sh}c'}\):
the available killing rate is too small to make the survival probability
vanish on this time scale. The minimax error is therefore at least
\(\frac{1}{2}c_{\varphi}e^{- C_{sh}c'} = :c > 0\). The eigenfunction
lower bound of Lemma C.2, proved in Appendix C, is the quantitative form
of the survival estimate invoked here. \end{proof}
\subsection{Many-feature lower bound: necessity of the logarithm}
Let \(A_{1},\ldots,A_{m} \subset M\) be pairwise well-separated regions,
each comparable to an intrinsic \(\varepsilon\)-ball, and write
\(\tau_{i} = \tau_{A_{i}}\) for their hitting times. In each \(A_{i}\)
place an optional local topological modification indicated by
\(\sigma_{i} \in \{ 0,1\}\). A domain in the family is indexed by
\(\sigma = \left( \sigma_{1},\ldots,\sigma_{m} \right) \in \{ 0,1\}^{m}\),
and the placements are chosen so that the degree-\(q\) homology of
\(M_{\sigma}\) encodes the bits as independent homological features (in
a planar domain, each present obstacle contributes one independent
\(H_{1}\)-class). All domains agree outside
\(A_{1} \cup \cdots \cup A_{m}\). Throughout this subsection,
\emph{well-separated} means that the pairwise intrinsic distances
between the \(A_{i}\) are bounded below by a separation parameter large
enough, relative to \(\varepsilon\), that the capacity-additivity defect
\(\theta\) of Lemma B.1 satisfies the smallness condition stated there;
this is arranged in the constructions by placing the feature regions at
mutual distances growing with the prescribed time horizon, exactly as in
the second-moment arguments of cover-time theory {[}4{]}. In particular
the condition is \(m\)- and \(T\)-dependent: the separation \(\varrho\)
of Lemma B.1 must grow with \(m\) --- and with the time horizon \(T\)
--- for the defect \(\theta\) to be small enough, and it is this growth
that caps the number of mutually well-separated \(\varepsilon\)-regions
at \({N_{\varepsilon}(M)}^{\alpha}\), \(\alpha < 1\), in Corollary 7.7.
The second-moment argument behind the lower bound requires that missing
two distinct targets be at most weakly positively correlated. We isolate
this as a named estimate; its proof, by approximate additivity of
Neumann capacity for well-separated small targets, is given in Appendix
B. The constants there are uniform in \(m\) only because the separation
in these families is taken to grow with \(m\) and \(T\), as just
described; this uniformity of \(C_{dec}\) is what underlies the
Paley-Zygmund step in Proposition 7.6.
\begin{lemma}[Pairwise decorrelation]
There exist
constants \(C_{dec} \geq 1\) and \(c'' > 0\), depending
only on the geometric and analytic data of the construction, such that
for the well-separated target families above, for all \(i \neq j\)
and all
\(T \leq c''\,\kappa_{d}(\varepsilon)^{- 1}\log m\),

\[\mathbb{P}\left( \tau_{i} > T,\ \tau_{j} > T \right) \leq C_{dec}\mspace{6mu}\mathbb{P}\left( \tau_{i} > T \right)\mspace{6mu}\mathbb{P}\left( \tau_{j} > T \right).\]
\end{lemma}
\begin{proposition}[Many-feature lower bound]
There exist
constants \(c,c' > 0\) and families
\(\{ M_{\sigma}\}_{\sigma \in \{ 0,1\}^{m}}\) as above such that

\begin{multline*}
\inf_{{\widehat{H}}_{q}}\max_{\sigma \in \{ 0,1\}^{m}}P_{M_{\sigma}}^{T}\left( {\widehat{H}}_{q} \neq H_{q}\left( M_{\sigma} \right) \right) \geq c\\
\text{whenever}\quad T \leq c'\left( \lambda_{1}(M)^{- 1} + \kappa_{d}(\varepsilon)^{- 1}\log m \right).
\end{multline*}
\end{proposition}
\begin{proof} If the trajectory has not visited \(A_{i}\), the two
alternatives differing only in the bit \(\sigma_{i}\) generate identical
observations, so \(\sigma_{i}\) cannot be reliably recovered. Let
\[Z_{T} = \sum_{i = 1}^{m}\mathbf{1}_{\{\tau_{i} > T\}}\]
count the unvisited regions. Since
\(\lambda_{A_{i}} \asymp \kappa_{d}(\varepsilon)\), the one-target
survival probability satisfies
\(\mathbb{P}\left( \tau_{i} > T \right) \geq e^{- C_{1}\kappa_{d}(\varepsilon)T}\),
so for \(T \leq c'\,\kappa_{d}(\varepsilon)^{- 1}\log m\) with
\(c' < 1/C_{1}\),
\[\mathbb{E}\left\lbrack Z_{T} \right\rbrack \geq m\, e^{- C_{1}\kappa_{d}(\varepsilon)T} \geq m^{1 - C_{1}c'} \rightarrow \infty\quad\quad(m \rightarrow \infty).\]
Lemma 7.5 controls the cross terms in the second moment:
\(\mathbb{E}\left\lbrack Z_{T}^{2} \right\rbrack\mathbb{\leq E}\left\lbrack Z_{T} \right\rbrack + C_{dec}\left( \mathbb{E}\left\lbrack Z_{T} \right\rbrack \right)^{2} \lesssim \left( \mathbb{E}\left\lbrack Z_{T} \right\rbrack \right)^{2}\)
once \(\mathbb{E}\left\lbrack Z_{T} \right\rbrack \geq 1\) (the
computation is recorded in Appendix B), and the Paley-Zygmund inequality
yields \(\mathbb{P}\left( Z_{T} > 0 \right) \geq c_{0} > 0\). On
\(\{ Z_{T} > 0\}\) at least one homological bit is unresolved, and the
indistinguishability argument of Lemma 7.1, applied to the pair
differing in that bit, bounds the minimax error below by a positive
constant. \end{proof}
\begin{corollary}[Full-cover lower bound]
Taking
\(m \asymp {N_{\varepsilon}(M)}^{\alpha}\) well-separated
\(\varepsilon\)-scale regions distributed throughout the domain,
universal homological recovery with uniformly small error over the
corresponding class is impossible whenever

\[T \leq c\,\kappa_{d}(\varepsilon)^{- 1}\log N_{\varepsilon}(M),\]

i.e.~below the scale \(\left( \log\varepsilon^{- 1} \right)^{2}\)
in dimension \(2\) and
\(\varepsilon^{2 - d}\log\varepsilon^{- 1}\) in dimensions
\(d \geq 3\) --- precisely the universal scales of Corollary 6.4.
\end{corollary}
Here \(\alpha \in (0,1)\) is a fixed exponent: well-separation in the
sense of Lemma 7.5 limits the number of mutually decorrelated targets to
a power \(N_{\varepsilon}(M)^{\alpha}\) rather than
\(N_{\varepsilon}(M)\) itself, but
\(\log m \asymp \log N_{\varepsilon}(M)\), so the obstruction holds at the
universal scale of Corollary 6.4. The constant in the coupon-collector
term is correspondingly not sharp; matching it, which amounts to taking
\(m \asymp N_{\varepsilon}(M)\), is the content of the Brownian
cover-time theory {[}4{]} and is not claimed here.
\subsection{Combined lower bound}
\begin{theorem}[Combined dumbbell-feature lower bound]
There exist admissible dumbbell domains \(M = L \cup N \cup R\)
with \(m\) well-separated, hidden
\(\varepsilon\)-scale feature regions
\(A_{1},\ldots,A_{m} \subset R\), indexed by
\(\sigma \in \{ 0,1\}^{m}\) as in Proposition 7.6, with the
process started in \(L\), such that

\begin{multline*}
\inf_{{\widehat{H}}_{q}}\max_{\sigma \in \{ 0,1\}^{m}}P_{M_{\sigma}}^{T}\left( {\widehat{H}}_{q} \neq H_{q}\left( M_{\sigma} \right) \right) \geq c\\
\text{whenever}\quad T \leq c'\left( \lambda_{1}(M)^{- 1} + \kappa_{d}(\varepsilon)^{- 1}\log m \right).
\end{multline*}
\end{theorem}
\begin{proof} We use the dumbbell \(M = L \cup N^{\eta} \cup R\) of
Proposition 7.3, now carrying \(m\) pairwise well-separated
\(\varepsilon\)-scale feature regions \(A_{1},\ldots,A_{m} \subset R\)
indexed by \(\sigma \in \{ 0,1\}^{m}\) exactly as in Proposition 7.6,
with the process started at \(x_{0} \in L_{0}\). All domains agree on
\(M\backslash\left( A_{1} \cup \ldots \cup A_{m} \right)\); in
particular they agree on \(L \cup N^{\eta}\), so the killed-on-\(R\)
quantities \(\lambda_{R}\) and the survival bound of Lemma C.3 are those
of the feature-free dumbbell and are unaffected by the features.
Since
\(\lambda_{1}(M)^{- 1} + \kappa_{d}(\varepsilon)^{- 1}\log m \leq 2\max\{\lambda_{1}(M)^{- 1},\,\kappa_{d}(\varepsilon)^{- 1}\log m\}\),
any \(T\) with
\(T \leq c'\left( \lambda_{1}(M)^{- 1} + \kappa_{d}(\varepsilon)^{- 1}\log m \right)\)
satisfies at least one of \(T \leq 2c'\lambda_{1}(M)^{- 1}\) or
\(T \leq 2c'\kappa_{d}(\varepsilon)^{- 1}\log m\). We bound the minimax
error below in each case; taking \(c'\) smaller than the two thresholds
appearing below makes both bounds effective.
\emph{Case 1:} \(T \leq 2c'\lambda_{1}(M)^{- 1}\)
\emph{(bottleneck-dominated).} By Lemma C.1 and the dumbbell asymptotics
\(\lambda_{1}(M) \asymp \eta^{d - 1}\) one has
\(\lambda_{R} \leq C_{2}\lambda_{1}(M)\), and Lemma C.3 gives
\(P_{x_{0}}\left( \tau_{R} > T \right) \geq c_{\varphi}e^{- \lambda_{R}T} \geq c_{\varphi}e^{- 2C_{2}c'} \geq c_{1} > 0\).
On \(\{\tau_{R} > T\}\) no feature region is visited, since every
\(A_{i} \subset R\); hence every bit \(\sigma_{i}\) is unresolved, and
applying Lemma 7.1 to any two configurations differing in a single bit
gives minimax error at least \(c_{1}/2\).
\emph{Case 2:} \(T \leq 2c'\kappa_{d}(\varepsilon)^{- 1}\log m\)
\emph{(feature-dominated).} The feature regions are visited only while
the process is in \(R\). Let \(\ell_{R}(t)\) be the occupation time of
\(R\) up to time \(t\), and let \(Y\) be the time change of \(X\) by the
right-continuous inverse of \(\ell_{R}\) --- the process obtained by
observing \(X\) only during its sojourns in \(R\), with law from \(y\)
written \(P_{y}^{R}\). By the time-change theory for symmetric Dirichlet
forms {[}13{]}, \(Y\) is the \(\mu\)-symmetric Hunt process on
\(\overline{R}\) whose Dirichlet form is the trace of
\(\left( \mathcal{E},H^{1}(M) \right)\) on \(L^{2}(R,\mu)\); its
strongly local part is the form of reflected Brownian motion in \(R\)
with the Neumann condition on \(\partial R\), while the trace
contributes in addition only a non-local part carried by the interface
\(\partial R \cap N_{\eta}\). We do not identify \(Y\) with reflected
Brownian motion in \(R\); we use only that the two agree on the small
interior targets, as follows.
Every feature region \(A_{i}\) lies in the interior of \(R\), at
distance at least \(\rho_{0}\) from \(\partial R \cap N_{\eta}\). For
\(u\) vanishing quasi-everywhere on \(A_{i}\) the non-local part of the
trace form is nonnegative, so the principal eigenvalue
\(\lambda_{Y}\left( A_{i} \right)\) of \(Y\) killed on \(A_{i}\) is at
least the killed eigenvalue of reflected Brownian motion in \(R\), which
is \(\asymp \kappa_{d}(\varepsilon)\) by Theorem 5.3 applied in the
fixed chamber \(R\); and testing against \(1 - w_{A_{i}}\), with
\(w_{A_{i}}\) the Neumann capacitary potential of \(A_{i}\) supported in
a neighbourhood of \(A_{i}\) disjoint from \(\partial R \cap N_{\eta}\)
--- on which the non-local part vanishes, since \(1 - w_{A_{i}}\) is
constant there --- bounds it above by
\(\mathcal{E}\left( 1 - w_{A_{i}},1 - w_{A_{i}} \right) = \operatorname{Cap}_{N}\left( A_{i} \right) \asymp \kappa_{d}(\varepsilon)\).
Hence
\(\lambda_{Y}\left( A_{i} \right) \asymp \kappa_{d}(\varepsilon)\), with
constants depending only on the fixed chamber \(R\).
The same localisation shows that the Neumann capacities, the
small-target eigenfunction estimate of Lemma C.2, the pairwise
decorrelation of Lemma 7.5 and the second-moment bound of Appendix B for
the family \(\{ A_{i}\}\) are, up to constants depending only on \(R\),
those of reflected Brownian motion in \(R\): the relevant capacitary
potentials and eigenfunctions are supported, respectively near-constant,
away from \(\partial R \cap N_{\eta}\), where the non-local part of the
trace form acts. The one global ingredient of Lemma 7.5 not supplied by
this localisation --- the heat-kernel domination of Assumption 2.6,
entering through Lemma 4.2 --- also holds for \(Y\) with constants
depending only on \(R\). The strongly local part of the trace Dirichlet
form of \(Y\) is the Dirichlet form of reflected Brownian motion in the
fixed chamber \(R\) with the Neumann condition on \(\partial R\), and
the non-local part is nonnegative, so the trace form dominates it; the
Nash inequality of the smooth bounded chamber \(R\) therefore passes to
the trace form, and by the equivalence of Nash inequalities with
on-diagonal heat-kernel bounds {[}14{]} the transition density of \(Y\)
is bounded by a constant after a fixed time, both depending only on
\(R\). Since the mass \(\mu(R) = |R|/\left| M_{\eta} \right|\) stays
bounded away from \(0\) and \(\infty\) as the neck closes, this
domination is uniform in \(\eta\), so Lemma 4.2, and hence the
decorrelation of Lemma 7.5, hold for the family \(\{ A_{i}\}\) under the
law of \(Y\). Proposition 7.6 therefore applies to \(Y\), with all
constants depending only on \(R\) and not on \(\eta\). Since
\(\ell_{R}(T) \leq T\), the set of features visited by the dumbbell
motion within \(\lbrack 0,T\rbrack\) is the set visited by \(Y\) within
its running time \(\ell_{R}(T) \leq T\); consequently
\begin{multline*}
P_{x_{0}}\left( \text{all }A_{i}\text{ visited by }T \right)\\
\leq \sup_{y \in R}P_{y}^{R}\left( \text{all }A_{i}\text{ visited by }T \right).
\end{multline*}
By Proposition 7.6 applied to \(Y\) there is \(c_{4} \in (0,1)\) with
\[\sup_{y \in R}P_{y}^{R}\left( \text{all }A_{i}\text{ visited by }T \right) \leq 1 - c_{4}\]
whenever \(T \leq c_{3}\,\kappa_{d}(\varepsilon)^{- 1}\log m\). Taking
\(2c' \leq c_{3}\), with probability at least \(c_{4}\) some feature
region is unvisited at time \(T\); the corresponding bit is unresolved,
and Lemma 7.1 bounds the minimax error below by a positive constant.
In either case the minimax error is bounded below by a fixed positive
constant, which proves the claim. The two obstructions thus enter
through disjoint time-scale regimes rather than sequentially: the
additive threshold
\(\lambda_{1}(M)^{- 1} + \kappa_{d}(\varepsilon)^{- 1}\log m\) is
comparable to the maximum of the two, and whichever term dominates
supplies the obstruction. \end{proof}
Together, the lower bounds match the structure of the upper bounds: the
burn-in term cannot generally be removed (bottlenecks force
\(\lambda_{1}(M)^{- 1}\), and \(s_{0} \asymp \lambda_{1}(M)^{- 1}\) in
typical settings), the inverse-capacity factor cannot be improved (one
hidden feature already forces it), and the logarithm cannot be removed
in full-cover or exact-recovery regimes; the feature-adaptive bound of
Corollary 6.2 is likewise sharp in its dependence on \(m\) when the
certificate regions correspond to independent topological alternatives.
Optimality throughout is meant at the level of order, up to
multiplicative constants: only the order of the recovery time is
matched, the leading constant being pinned --- for the killed eigenvalue
alone --- only under the additional regularity of Appendix D. In
particular the matching of the spectral-access term is conditional on
the identification \(s_{0} \asymp \lambda_{1}(M)^{- 1}\) recorded after
Assumption 2.6.
\section{Persistence and discrete-observation consequences}
The recovery theorems have two consequences useful in applications; both
are routine translations of the density statement following Theorem 6.3,
and neither changes the recovery scale.
The first concerns persistence. If \(d_{H}(A,B) \leq \eta\) for compact
\(A,B \subset \mathbb{R}^{d}\), then
\(A^{\rho} \subseteq B^{\rho + \eta}\) and
\(B^{\rho} \subseteq A^{\rho + \eta}\) for every \(\rho \geq 0\), so the
offset filtrations are \(\eta\)-interleaved, and bottleneck stability
for persistence diagrams {[}9, 10{]} gives
\(d_{B}\left( {Dgm}_{q}(A),{Dgm}_{q}(B) \right) \leq \eta\) in every
degree.
\begin{corollary}[Persistence approximation]
There exist
constants \(C > 0\) and \(\varepsilon_{0} > 0\) such that
for every \(0 < \varepsilon < \varepsilon_{0}\),
\(\delta \in (0,1)\), \(z \in M\), and every
\(T \geq C\left( s_{0} + \kappa_{d}(\varepsilon)^{- 1}\left( \log N_{\varepsilon}(M) + \log\frac{1}{\delta} \right) \right)\),
with probability at least \(1 - \delta\),

\[d_{B}\left( {Dgm}_{q}\left( \Gamma_{T} \right),{Dgm}_{q}(M) \right) \leq C\varepsilon\quad\quad\text{for every degree }q.\]
\end{corollary}
\begin{proof} On the density event of Theorem 6.3,
\(d_{H}\left( \Gamma_{T},M \right) \leq C\varepsilon\) by Lemma 2.4; the
interleaving observation and bottleneck stability conclude. \end{proof}
The second concerns discrete observations. Suppose the trajectory is
sampled at mesh \(\Delta\), producing
\(\Gamma_{T}^{\Delta} = \{ X_{k\Delta}:0 \leq k\Delta \leq T\}\), and
let
\(\omega_{X}(\Delta) = \sup\{\left| X_{t} - X_{s} \right|:|t - s| \leq \Delta,\ 0 \leq s,t \leq T\}\)
denote the path modulus.
\begin{corollary}[Recovery from sampled trajectories]
There
exist constants \(C,C_{low},C_{up},\varepsilon_{0},c > 0\) such
that if \(0 < \varepsilon < \varepsilon_{0}\),
\[T \geq C\left( s_{0} + \kappa_{d}(\varepsilon)^{- 1}\left( \log N_{\varepsilon}(M) + \log\frac{1}{\delta} \right) \right)\],
and \(\omega_{X}(\Delta) \leq c\varepsilon\), then with
probability at least \(1 - \delta\),

\begin{multline*}
H_{q}\left( \left( \Gamma_{T}^{\Delta} \right)^{r} \right) \cong H_{q}(M)\\
\text{for every }q = 0,\ldots,d\text{ and every }C_{low}\,\varepsilon \leq r \leq C_{up}\, \operatorname{reach}(M).
\end{multline*}
\end{corollary}
\begin{proof} On the density event,
\(d_{H}\left( \Gamma_{T},M \right) \leq C\varepsilon\); on
\(\{\omega_{X}(\Delta) \leq c\varepsilon\}\), also
\(d_{H}\left( \Gamma_{T}^{\Delta},\Gamma_{T} \right) \leq c\varepsilon\),
hence \(d_{H}\left( \Gamma_{T}^{\Delta},M \right) \leq C'\varepsilon\),
and Theorem 3.1 applies after adjusting constants. \end{proof}
By the standard modulus of continuity of (reflected) Brownian motion,
the mesh condition holds with high probability once
\(\sqrt{\Delta\log(T/\Delta)} \lesssim \varepsilon\); it is an
observation-resolution condition, not an exploration condition. Finally,
since the sampled trace is a finite point set, the Nerve theorem
{[}24{]} identifies the homology of the Čech complex
\({\check{C}}_{r}\left( \Gamma_{T}^{\Delta} \right)\) with that of the
offset \(\left( \Gamma_{T}^{\Delta} \right)^{r}\), so Corollary 8.2
transfers verbatim to Čech complexes, and via the standard
Čech-Vietoris-Rips interleaving to Vietoris-Rips complexes at a
comparable scale, provided the reconstruction window is wide enough that
the interleaved scales remain below the reach threshold.
Together these corollaries make the procedure data-driven. One computes
the persistent homology of the Čech or Vietoris-Rips filtration of the
observed sample and reads the topology of the unknown domain off the
resulting barcode; and because Corollary 8.2 holds simultaneously for
every offset from the resolution scale up to a fixed fraction of the
reach, the offset itself need not be known in advance, but is selected
from the data --- as the stable plateau of scales on which the recovered
Betti numbers are constant --- rather than supplied from prior knowledge
of the geometry.
\appendix
\section{Capacity estimates for small reflected targets}
This appendix proves the local condenser estimates used in Proposition
5.2, including the boundary case. The analytic chain is: Euclidean model
condensers have capacity scale \(\kappa_{d}(\varepsilon)\); Neumann
reflection turns half-space condensers into full-space condensers with
comparable energy; and uniform bi-Lipschitz charts transfer both
estimates to local condensers in \(M\).
\subsection{Euclidean model condensers}
For \(0 < \varepsilon < \rho\), the condenser capacity of
\(\left( B(0,\varepsilon),B(0,\rho) \right) \subset \mathbb{R}^{d}\) is
\begin{multline*}
\operatorname{Cap}_{\mathbb{R}^{d}}\left( B(0,\varepsilon),B(0,\rho) \right) = \inf\Big\{ \tfrac{1}{2}\int_{B(0,\rho)}^{}|\nabla u|^{2}\, dx:\\
u \in H_{0}^{1}\left( B(0,\rho) \right),\ u \geq 1\text{ on }B(0,\varepsilon) \Big\}.
\end{multline*}
\begin{lemma}[Euclidean condenser capacity]
For fixed
\(\rho > 0\), as \(\varepsilon \downarrow 0\),

\[\operatorname{Cap}_{\mathbb{R}^{d}}\left( B(0,\varepsilon),B(0,\rho) \right) \asymp \kappa_{d}(\varepsilon);\]

explicitly,
\(\asymp \left( \log(\rho/\varepsilon) \right)^{- 1}\) for
\(d = 2\) and \(\asymp \varepsilon^{d - 2}\) for
\(d \geq 3\).
\end{lemma}
\begin{proof} The minimizer is radial and harmonic in the annulus
\(B(0,\rho)\backslash B(0,\varepsilon)\). For \(d \geq 3\) it is
\(u(r) = \frac{r^{2 - d} - \rho^{2 - d}}{\varepsilon^{2 - d} - \rho^{2 - d}}\),
and a direct computation gives
\[\int_{B(0,\rho)\backslash B(0,\varepsilon)}^{}|\nabla u|^{2}\, dx = c_{d}\left( \varepsilon^{2 - d} - \rho^{2 - d} \right)^{- 1} \asymp \varepsilon^{d - 2}\]
for fixed \(\rho\) and small \(\varepsilon\). For \(d = 2\) the
minimizer is \(u(r) = \frac{\log(\rho/r)}{\log(\rho/\varepsilon)}\),
with
\(\left| \nabla u(r) \right| = \left( r\log(\rho/\varepsilon) \right)^{- 1}\),
hence
\[\int_{B(0,\rho)\backslash B(0,\varepsilon)}^{}|\nabla u|^{2}\, dx = \frac{2\pi}{\log(\rho/\varepsilon)},\]
which for fixed \(\rho\) is comparable to
\(\left( \log\varepsilon^{- 1} \right)^{- 1}\). \end{proof}
\subsection{Half-space condensers and Neumann reflection}
Boundary targets are compared with a half-space model. Let
\(\mathbb{H}^{d} = \{ x \in \mathbb{R}^{d}:x_{d} \geq 0\}\),
\(B^{+}(0,s) = B(0,s) \cap \mathbb{H}^{d}\), and define the Neumann
half-space condenser capacity
\(\operatorname{Cap}_{\mathbb{H}^{d},N}\left( B^{+}(0,\varepsilon),B^{+}(0,\rho) \right)\)
as the infimum of
\(\frac{1}{2}\int_{B^{+}(0,\rho)}^{}|\nabla u|^{2}\, dx\) over functions
\(u \geq 1\) on \(B^{+}(0,\varepsilon)\) vanishing on the spherical part
of the boundary, with no constraint on the flat face \(\{ x_{d} = 0\}\)
(the Neumann condition).
\begin{lemma}[Reflection does not change the capacity scale]
For fixed \(\rho > 0\), as
\(\varepsilon \downarrow 0\),

\[\operatorname{Cap}_{\mathbb{H}^{d},N}\left( B^{+}(0,\varepsilon),B^{+}(0,\rho) \right) \asymp \kappa_{d}(\varepsilon).\]
\end{lemma}
\begin{proof} Given an admissible \(u\) on \(B^{+}(0,\rho)\), its even
extension
\(\widetilde{u}\left( x',x_{d} \right) = u\left( x',\left| x_{d} \right| \right)\)
is admissible for the full-ball condenser (up to harmless changes of
target and outer boundary) with
\(\int_{B(0,\rho)}^{}\left| \nabla\widetilde{u} \right|^{2}\, dx = 2\int_{B^{+}(0,\rho)}^{}|\nabla u|^{2}\, dx\).
Conversely, any radial admissible function for the full-ball condenser
restricts to an admissible function for the half-ball Neumann condenser
with half the energy. Hence the two capacities differ by a
multiplicative constant, and Lemma A.1 gives the scale. \end{proof}
In probabilistic language: reflected Brownian motion near a Neumann
boundary is, after even reflection, comparable to ordinary Brownian
motion in full space, so a target touching the reflecting boundary has
the same capacity order as an interior target. This uniformity is what
allows intrinsic nets --- which necessarily contain boundary-adjacent
balls --- to be treated with a single capacity scale. The situation
differs sharply from Dirichlet (absorbing) boundaries.
\subsection{Local charts and the local condenser estimate}
The admissibility assumptions on \(M\) (Assumption 2.5) provide
constants \(\rho_{ch} > 0\) and \(L_{ch} \geq 1\) and, for every
\(x \in M\), a chart
\(\Psi_{x}:U_{x} \rightarrow V_{x} \subset \mathbb{R}^{d}\) on a
neighborhood \(U_{x} \subset M\) of \(x\) such that: if
\(\operatorname{dist}(x,\partial M) \geq \rho_{ch}\), then \(\Psi_{x}\) maps \(U_{x}\)
bi-Lipschitzly onto a Euclidean ball; if
\(\operatorname{dist}(x,\partial M) < \rho_{ch}\), then \(\Psi_{x}\) maps \(U_{x}\)
bi-Lipschitzly onto a half-ball with the reflecting boundary carried
into the flat face; and in both cases
\[L_{ch}^{- 1}|y - z| \leq \left| \Psi_{x}(y) - \Psi_{x}(z) \right| \leq L_{ch}|y - z|,\quad\quad y,z \in U_{x},\]
with Jacobian determinants of \(\Psi_{x}\) and \(\Psi_{x}^{- 1}\)
bounded above and below by constants of the admissible geometry. The
charts are used only at the fixed geometric scale \(\rho\); the small
parameter is \(\varepsilon\).
For \(x \in M\), \(0 < \rho < \rho_{ch}\), and
\(0 < \varepsilon \ll \rho\), recall
\begin{multline*}
\operatorname{Cap}_{loc}(x,\varepsilon;\rho) = \inf\Big\{ \tfrac{1}{2}\int_{B_{M}(x,\rho)}^{}|\nabla u|^{2}\, d\mu:\\
u \geq 1\text{ on }B_{M}(x,\varepsilon),\ u = 0\text{ on }\partial B_{M}(x,\rho) \cap M \Big\},
\end{multline*}
with Neumann condition on \(\partial M\) inside the patch.
\begin{proposition}[Uniform local condenser estimate]
There
exist constants \(c_{loc},C_{loc} > 0\),
\(\rho_{0} > 0\), \(\varepsilon_{0} > 0\), depending only
on the admissible geometry of \(M\), such that for every
\(x \in M\), every \(0 < \rho \leq \rho_{0}\), and every
\(0 < \varepsilon < \varepsilon_{0}\rho\),

\[c_{loc}\,\kappa_{d}(\varepsilon) \leq \operatorname{Cap}_{loc}(x,\varepsilon;\rho) \leq C_{loc}\,\kappa_{d}(\varepsilon).\]
\end{proposition}
\begin{proof} Use the chart \(\Psi_{x}\). If \(x\) is an interior point,
\(\Psi_{x}\) maps the local condenser to a Euclidean condenser between
sets comparable to balls of radii \(\varepsilon\) and \(\rho\); the
bi-Lipschitz and Jacobian bounds distort Dirichlet energies by at most a
multiplicative constant depending only on \(L_{ch}\), so the local
capacity is comparable to the Euclidean condenser capacity and Lemma A.1
gives the scale. If \(x\) is near \(\partial M\), \(\Psi_{x}\) maps the
local geometry to a half-ball with Neumann condition on the flat face;
the same energy comparison applies and Lemma A.2 gives the same scale.
Uniformity in \(x\) follows from compactness of \(M\) and the uniform
chart constants. \end{proof}
Proposition A.3 is the input quoted in the proof of Proposition 5.2,
which converts it into the global two-sided estimate
\(\operatorname{Cap}_{N}\left( B_{M}(x,\varepsilon) \right) \asymp \kappa_{d}(\varepsilon)\)
and thence, through Proposition 5.1, into the uniform small-hole
eigenvalue Theorem 5.3.
\section{Decorrelation and the second-moment supplement}
This appendix proves the pairwise decorrelation estimate of Lemma 7.5
and records the second-moment computation used in Proposition 7.6 and
Corollary 7.7: below the capacitary coupon-collector time
\(\kappa_{d}(\varepsilon)^{- 1}\log m\), at least one feature region
remains unvisited with probability bounded away from zero. The argument
is the first/second-moment method; the problem-specific input is
approximate additivity of Neumann capacity for well-separated small
targets, which we isolate first.
\subsection{Approximate additivity of capacity}
\begin{lemma}[Approximate capacity additivity]
Let
\(A_{i},A_{j} \subset M\) be targets, each comparable to an
intrinsic \(\varepsilon\)-ball, whose intrinsic distance is at
least a separation parameter \(\varrho > 0\) with
\(\varepsilon < \varepsilon_{0}\varrho\). Then there exists
\(\theta = \theta(\varepsilon/\varrho) \in \lbrack 0,1)\), with
\(\theta(t) \rightarrow 0\) as \(t \rightarrow 0\), such
that

\[(1 - \theta)\left( \operatorname{Cap}_{N}\left( A_{i} \right) + \operatorname{Cap}_{N}\left( A_{j} \right) \right) \leq \operatorname{Cap}_{N}\left( A_{i} \cup A_{j} \right) \leq \operatorname{Cap}_{N}\left( A_{i} \right) + \operatorname{Cap}_{N}\left( A_{j} \right).\]

Consequently, by the comparison of Proposition 5.1, the killed
eigenvalues satisfy, for all sufficiently small \(\varepsilon\),

\[\lambda_{A_{i} \cup A_{j}} \geq c_{ce}\, C_{ce}^{- 1}(1 - \theta)\left( \lambda_{A_{i}} + \lambda_{A_{j}} \right).\]
\end{lemma}
\begin{proof} The subadditivity (right-hand) inequality is immediate: if
\(u_{i},u_{j}\) are admissible capacitary potentials for \(A_{i},A_{j}\)
built from the disjoint local condenser potentials of Proposition A.3
(supported in patches \(B_{M}( \cdot ,\varrho/2)\) around each target,
then corrected to mean zero as in the proof of Proposition 5.2), the
function \(\max\left( u_{i},u_{j} \right)\) is admissible for the union
after the same mean-zero correction, and its energy is at most the sum
of the energies because the supports of the gradients are disjoint.
For the superadditivity (left-hand) inequality, let \(u\) be admissible
for \(\operatorname{Cap}_{N}\left( A_{i} \cup A_{j} \right)\). Apply the localization
argument of the lower bound in Proposition 5.2 separately in the two
disjoint patches \(B_{M}\left( x_{i},\varrho/2 \right)\) and
\(B_{M}\left( x_{j},\varrho/2 \right)\) around the two targets: in each
patch, either the cutoff construction produces a local condenser
competitor of energy at most
\(C\,\mathcal{E}\left( u,u;\text{patch} \right)\), or the mean-zero
constraint forces a fixed energy \(c_{0} > 0\) that dominates both
capacities for small \(\varepsilon\). Since the patches are disjoint,
the two local energies add without double counting:
\begin{align*}
\mathcal{E}(u,u) &\geq \mathcal{E}\left( u,u;B_{M}\left( x_{i},\varrho/2 \right) \right) + \mathcal{E}\left( u,u;B_{M}\left( x_{j},\varrho/2 \right) \right)\\
&\geq c\left( \operatorname{Cap}_{loc}\left( x_{i},\varepsilon;\varrho/2 \right) + \operatorname{Cap}_{loc}\left( x_{j},\varepsilon;\varrho/2 \right) \right),
\end{align*}
and Proposition A.3 together with the upper bound of Proposition 5.2
converts the right-hand side into
\((1 - \theta)\left( \operatorname{Cap}_{N}\left( A_{i} \right) + \operatorname{Cap}_{N}\left( A_{j} \right) \right)\),
with a defect \(\theta\) controlled by the ratio of the chart-comparison
constants at scales \(\varepsilon\) and \(\varrho\); for the model
condensers of Lemma A.1, the defect is of order
\(\left( \log(\varrho/\varepsilon) \right)^{- 1}\) in dimension \(2\)
and of order \((\varepsilon/\varrho)^{d - 2}\) in dimensions
\(d \geq 3\), and in both cases
\(\theta(\varepsilon/\varrho) \rightarrow 0\) as
\(\varepsilon/\varrho \rightarrow 0\). The eigenvalue consequence
follows by applying Proposition 5.1 to \(A_{i} \cup A_{j}\) (whose
capacity is small for small \(\varepsilon\)) and to each \(A_{i}\)
separately. \end{proof}
\subsection{Proof of Lemma 7.5}
The joint survival event is the survival of the union:
\(\{\tau_{i} > T,\ \tau_{j} > T\} = \{\tau_{A_{i} \cup A_{j}} > T\}\).
By Lemma 4.2 applied to the union, and then Lemma B.1,
\[\mathbb{P}\left( \tau_{i} > T,\ \tau_{j} > T \right) \leq a_{hk}\, e^{- \lambda_{A_{i} \cup A_{j}}\left( T - s_{0} \right)} \leq a_{hk}\, e^{\left( \lambda_{A_{i}} + \lambda_{A_{j}} \right)s_{0}}\, e^{- (1 - \theta')\,\left( \lambda_{A_{i}} + \lambda_{A_{j}} \right)\, T},\]
where \(\theta' \in \lbrack 0,1)\) absorbs both the additivity defect
\(\theta\) and the fixed factor \(c_{ce}C_{ce}^{- 1}\) of Lemma B.1 into
a single defect that tends to the fixed value \(1 - c_{ce}C_{ce}^{- 1}\)
as \(\theta \rightarrow 0\); replacing the two-sided comparison
constants of Proposition 5.1 by the sharper asymptotic form
\(\lambda_{A} = \operatorname{Cap}_{N}(A)\left( 1 + O\left( \operatorname{Cap}_{N}(A) \right) \right)\)
stated there, valid in the small-capacity regime of the construction,
the factor \(c_{ce}C_{ce}^{- 1}\) improves to \(1 + o(1)\) as
\(\varepsilon \downarrow 0\), so that \(\theta' \rightarrow 0\) when
both \(\varepsilon \downarrow 0\) and
\(\varepsilon/\varrho \rightarrow 0\).
In the opposite direction, the one-target survival probabilities in the
lower-bound constructions satisfy the matching two-sided estimate
\[c_{0}\, e^{- (1 + \theta')\,\lambda_{A_{i}}T}\mathbb{\leq P}\left( \tau_{i} > T \right) \leq a_{hk}\, e^{- \lambda_{A_{i}}\left( T - s_{0} \right)},\]
for starting points at macroscopic distance from the targets: the lower
estimate is Lemma 7.2 applied to the process killed on each small
target, with rate within a factor \(1 + \theta'\) of \(\lambda_{A_{i}}\)
on the time horizon considered (Lemma 7.2, as in the proof of
Proposition 7.4). Combining the three displays,
\[\frac{\mathbb{P}\left( \tau_{i} > T,\ \tau_{j} > T \right)}{\mathbb{P}\left( \tau_{i} > T \right)\,\mathbb{P}\left( \tau_{j} > T \right)} \leq \frac{a_{hk}}{c_{0}^{2}}\mspace{6mu} e^{\left( \lambda_{A_{i}} + \lambda_{A_{j}} \right)s_{0}}\mspace{6mu} e^{2\theta'\left( \lambda_{A_{i}} + \lambda_{A_{j}} \right)T}.\]
Since
\(\lambda_{A_{i}} + \lambda_{A_{j}} \leq 2C_{sh}\,\kappa_{d}(\varepsilon)\)
by Theorem 5.3, the first exponential factor is
\(e^{O\left( \kappa_{d}(\varepsilon)s_{0} \right)} \rightarrow 1\),
while on the time window
\(T \leq c''\,\kappa_{d}(\varepsilon)^{- 1}\log m\) the second factor is
at most \(e^{4C_{sh}\theta'c''\log m} = m^{4C_{sh}\theta'c''}\). Choosing
the separation parameter \(\varrho\) in the construction so that
\(\theta' \leq \left( 4C_{sh}c''\log m \right)^{- 1}\) --- which is
possible because \(\theta' \rightarrow 0\) as
\(\varepsilon/\varrho \rightarrow 0\) and the constructions place the
feature regions at mutual distances chosen after \(m\) and the time
horizon are fixed --- the right-hand side is bounded by a constant
\(C_{dec}\) depending only on \(a_{hk}\), \(c_{0}\), and the fixed
geometric data. \hfill\(\blacksquare\)
\subsection{The second-moment computation}
Let \(Z_{T} = \sum_{i = 1}^{m}\mathbf{1}_{\{\tau_{i} > T\}}\) count the
unvisited regions, as in the proof of Proposition 7.6.

\textbf{First moment.} Since
\(\lambda_{A_{i}} \asymp \kappa_{d}(\varepsilon)\) (Theorem 5.3), the
survival lower bound above gives a constant \(C_{1} > 0\) with
\(\mathbb{P}\left( \tau_{i} > T \right) \geq e^{- C_{1}\kappa_{d}(\varepsilon)T}\)
for every \(i\) on the time scales considered. Hence
\[\mathbb{E}\left\lbrack Z_{T} \right\rbrack \geq m\, e^{- C_{1}\kappa_{d}(\varepsilon)T} \geq m^{1 - C_{1}c'}\quad\quad\text{for }T \leq c'\,\kappa_{d}(\varepsilon)^{- 1}\log m,\]
which diverges as \(m \rightarrow \infty\) once \(c' < 1/C_{1}\).

\textbf{Second moment.} Expanding the square and applying Lemma 7.5 to
the off-diagonal terms,
\begin{align*}
\mathbb{E}\left\lbrack Z_{T}^{2} \right\rbrack &= \mathbb{E}\left\lbrack Z_{T} \right\rbrack + \sum_{i \neq j}\mathbb{P}\left( \tau_{i} > T,\ \tau_{j} > T \right)\\
&\leq \mathbb{E}\left\lbrack Z_{T} \right\rbrack + C_{dec}\sum_{i \neq j}\mathbb{P}\left( \tau_{i} > T \right)\,\mathbb{P}\left( \tau_{j} > T \right)\\
&\leq \mathbb{E}\left\lbrack Z_{T} \right\rbrack + C_{dec}\left( \mathbb{E}\left\lbrack Z_{T} \right\rbrack \right)^{2},
\end{align*}
so
\(\mathbb{E}\left\lbrack Z_{T}^{2} \right\rbrack \leq \left( 1 + C_{dec} \right)\left( \mathbb{E}\left\lbrack Z_{T} \right\rbrack \right)^{2}\)
once \(\mathbb{E}\left\lbrack Z_{T} \right\rbrack \geq 1\).

\textbf{Conclusion.} The Paley-Zygmund inequality gives
\[\mathbb{P}\left( Z_{T} > 0 \right) \geq \frac{\left( \mathbb{E}\left\lbrack Z_{T} \right\rbrack \right)^{2}}{\mathbb{E}\left\lbrack Z_{T}^{2} \right\rbrack} \geq \frac{1}{1 + C_{dec}} > 0.\]
On the event \(\{ Z_{T} > 0\}\), at least one feature region is
unvisited and the corresponding topological bit is observationally
unresolved, which is the input used in the proofs of Proposition 7.6 and
Theorem 7.8.
\section{Eigenvalue and eigenfunction estimates for the lower-bound constructions}
This appendix supplies the analytic inputs used in Section 7: an upper
bound on the killed eigenvalue of a chamber lying behind a thin neck
(Lemma C.1), and lower bounds on the principal killed eigenfunction at
the starting point, for a small target (Lemma C.2) and for the far
chamber of a dumbbell (Lemma C.3). Throughout,
\(E(f,f) = \frac{1}{2}\int_{M}^{}|\nabla f|^{2}\, d\mu\) is the
Dirichlet form of Section 2 and \(\mu\) is normalized Lebesgue measure
on the relevant domain. We use repeatedly Lemma 7.2, in the form valid
for an arbitrary normalization: if \(B \subset M\) is closed of positive
capacity with principal killed eigenvalue \(\lambda_{B}\) and
\(\psi \geq 0\) is any principal eigenfunction of the killed problem,
then for every \(x\) and every \(t \geq 0\),
\[P_{x}\left( \tau_{B} > t \right)\  \geq \ \frac{\psi(x)}{\parallel \psi \parallel_{L^{\infty}}}\, e^{- \lambda_{B}t}.\]
The dumbbell chambers \(L,R\) of Lemmas C.1 and C.3, and the fixed
ambient domain carrying the small target of Lemma C.2, are taken smooth
--- uniformly of class \(C^{1,1}\) suffices --- so that the interior and
up-to-the-Neumann-boundary elliptic and Harnack estimates used in the
proofs of Lemmas C.2 and C.3 are available. This is a stronger boundary
regularity than the Lipschitz admissibility of Assumption 2.5 used in
the main theory; it is harmless here, since a minimax lower bound over a
smoother subclass of admissible domains is a fortiori a lower bound, and
the constructions are free to use smooth chambers.
\begin{lemma}[Killed eigenvalue behind a thin neck]
Let
\(M^{\eta} = L \cup N^{\eta} \cup R\) be the dumbbell of
Proposition 7.3, with chambers \(L,R\) fixed and neck
\(N^{\eta}\) of fixed length \(\ell\) and
cross-sectional \((d - 1)\)-measure of order
\(\eta^{d - 1}\). Then there is \(C_{1} > 0\), independent
of \(\eta\), such that the eigenvalue of the reflected generator
killed on \(R\) satisfies
\(\lambda_{R} \leq C_{1}\,\eta^{d - 1}\) for all small
\(\eta\).
\end{lemma}
\begin{proof} Parametrize \(N^{\eta}\) by arclength
\(s \in \left( 0\mathcal{,l} \right)\) along its axis, and define
\(f:M^{\eta} \rightarrow \lbrack 0,1\rbrack\) by \(f \equiv 1\) on
\(L\), \(f(s, \cdot ) = 1 - s\mathcal{/l}\) on \(N^{\eta}\), and
\(f \equiv 0\) on \(R\). Then \(f \in H^{1}\left( M^{\eta} \right)\) and
\(f = 0\) quasi-everywhere on \(R\), so \(f\) is admissible for the
variational characterization of \(\lambda_{R}\). Its gradient is
supported in the neck, where \(|\nabla f| = 1\mathcal{/l}\), so
\(\int_{M^{\eta}}^{}|\nabla f|^{2}\, dx = \ell^{- 2}\,\left| N^{\eta} \right| \leq C\,\eta^{d - 1}\ell^{- 1}\),
whence
\(E(f,f) = \frac{1}{2}\left| M^{\eta} \right|^{- 1}\int|\nabla f|^{2}\, dx \leq C\,\eta^{d - 1}\).
Since \(f \equiv 1\) on \(L\),
\(\parallel f \parallel_{L^{2}(\mu)}^{2} \geq \mu(L) \geq c_{0} > 0\).
Therefore
\(\lambda_{R} \leq E(f,f)/ \parallel f \parallel_{L^{2}(\mu)}^{2} \leq C_{1}\eta^{d - 1}\).
\end{proof}
\begin{lemma}[Eigenfunction lower bound for a small target]
There exist \(\rho_{0} > 0\), \(c_{\varphi} > 0\)
and \(\varepsilon_{1} > 0\), depending only on the data of
the fixed ambient domain \(M\), such that for every centre
\(x_{*} \in M\) and every
\(\varepsilon < \varepsilon_{1}\), the process killed on
\(A_{\varepsilon} = B_{M}\left( x_{*},\varepsilon \right)\)
satisfies, for all \(t \geq 0\) and all starting points
\(x\) with
\(d_{M}\left( x,A_{\varepsilon} \right) \geq \rho_{0}\),

\[P_{x}\left( \tau_{A_{\varepsilon}} > t \right) \geq c_{\varphi}\, e^{- \lambda_{A_{\varepsilon}}t}.\]
\end{lemma}
\begin{proof} Let \(\Phi \geq 0\) be the principal eigenfunction of the
process killed on \(A_{\varepsilon}\), normalized by
\(\parallel \Phi \parallel_{L^{2}(\mu)} = 1\), so
\(\frac{1}{2}\Delta\Phi = - \lambda_{A_{\varepsilon}}\Phi\) on
\(M\backslash A_{\varepsilon}\) with the Neumann condition on
\(\partial M\) and \(\Phi = 0\) quasi-everywhere on \(A_{\varepsilon}\).
By Theorem 5.3,
\(\lambda_{A_{\varepsilon}} \leq C_{sh}\,\kappa_{d}(\varepsilon) \rightarrow 0\),
so \(\lambda_{A_{\varepsilon}} \leq \Lambda\) for \(\varepsilon\) small.
The mean \(m = \int_{M}^{}\Phi\, d\mu\) and the oscillation of \(\Phi\)
are controlled exactly as in the proof of Proposition 5.1:
\(m \geq \frac{1}{2}\) and
\(\parallel \Phi - m \parallel_{L^{2}(\mu)}^{2} \leq \lambda_{A_{\varepsilon}}/\lambda_{1}(M)\).
On the set
\(K = \{ x:d_{M}\left( x,A_{\varepsilon} \right) \geq \rho_{0}\}\) the
function \(\Phi - m\) solves
\(\frac{1}{2}\Delta(\Phi - m) + \lambda_{A_{\varepsilon}}(\Phi - m) = - \lambda_{A_{\varepsilon}}m\),
an equation with right-hand side bounded by \(\Lambda\); the uniform
local elliptic estimate available under Assumption 2.5 (interior and up
to the Neumann boundary) gives
\(\parallel \Phi - m \parallel_{L^{\infty}(K)} \leq C\left( \parallel \Phi - m \parallel_{L^{2}(\mu)} + \lambda_{A_{\varepsilon}} \right) \leq C\left( \sqrt{\lambda_{A_{\varepsilon}}/\lambda_{1}(M)} + \lambda_{A_{\varepsilon}} \right)\),
which tends to \(0\) as \(\varepsilon \rightarrow 0\). Hence
\(\Phi \geq m - o(1) \geq \frac{1}{4}\) on \(K\) for
\(\varepsilon < \varepsilon_{1}\). Moreover the killed heat kernel is
dominated by the Neumann heat kernel, so by Assumption 2.6
\(\parallel \Phi \parallel_{L^{\infty}} \leq C_{0} \parallel \Phi \parallel_{L^{2}(\mu)} = C_{0}\)
uniformly. Inserting \(\psi = \Phi\) into the survival inequality above
gives
\(P_{x}\left( \tau_{A_{\varepsilon}} > t \right) \geq \left( \Phi(x)/ \parallel \Phi \parallel_{L^{\infty}} \right)e^{- \lambda_{A_{\varepsilon}}t} \geq \left( 4C_{0} \right)^{- 1}e^{- \lambda_{A_{\varepsilon}}t}\)
for \(x \in K\). \end{proof}
\begin{lemma}[Eigenfunction lower bound for the far chamber]
With \(M^{\eta} = L \cup N^{\eta} \cup R\) as in Lemma C.1,
there exist \(c_{\varphi} > 0\) and \(\eta_{1} > 0\),
independent of \(\eta\), such that the process killed on \(R\)
satisfies, for every \(x_{0}\) in a fixed interior bulk
\(L_{0} \subset \subset L\) and every \(t \geq 0\),

\[P_{x_{0}}\left( \tau_{R} > t \right) \geq c_{\varphi}\, e^{- \lambda_{R}t}\quad\quad\left( \eta < \eta_{1} \right).\]
\end{lemma}
\begin{proof} Let \(\Phi \geq 0\) be the principal eigenfunction of the
process killed on \(R\), on \(M^{\eta}\backslash R = L \cup N^{\eta}\),
normalized \(\parallel \Phi \parallel_{L^{2}(\mu)} = 1\); thus
\(\frac{1}{2}\Delta\Phi = - \lambda_{R}\Phi\) with the Neumann condition
on the reflecting boundary and \(\Phi = 0\) quasi-everywhere on the
interface \(\partial R\). By Lemma C.1,
\(\lambda_{R} \leq C_{1}\eta^{d - 1} \leq \Lambda\). The killed heat
kernel is dominated by the Neumann heat kernel of \(M^{\eta}\), so
\(\parallel \Phi \parallel_{L^{\infty}} \leq C_{0}\) uniformly in
\(\eta\). Since \(\left| N^{\eta} \right| \leq C\eta^{d - 1}\), one has
\(\int_{N^{\eta}}^{}\Phi^{2}\, d\mu \leq C\eta^{d - 1} \parallel \Phi \parallel_{L^{\infty}}^{2} \leq C'\eta^{d - 1}\),
whence
\(\int_{L}^{}\Phi^{2}\, d\mu \geq 1 - C'\eta^{d - 1} \geq \frac{3}{4}\)
for \(\eta < \eta_{1}\). Fix the interior bulk
\(L_{0} \subset \subset L\) so that
\(\mu\left( L\backslash L_{0} \right) \leq \left( 4C_{0}^{2} \right)^{- 1}\);
then
\(\int_{L_{0}}^{}\Phi^{2}\, d\mu \geq \frac{3}{4} - \mu\left( L\backslash L_{0} \right)C_{0}^{2} \geq \frac{1}{2}\),
and therefore
\(\sup_{L_{0}}\Phi \geq (\int_{L_{0}}^{}\Phi^{2}\, d\mu/\mu\left( L_{0} \right))^{1/2} \geq \left( \frac{1}{2} \right)^{1/2} \geq \frac{1}{2}\),
using \(\mu\left( L_{0} \right) \leq 1\). On the fixed connected open
set \(L\), \(\Phi\) is a nonnegative solution of
\(\frac{1}{2}\Delta\Phi + \lambda_{R}\Phi = 0\) with the Neumann
condition on the reflecting part of \(\partial L\); the elliptic Harnack
inequality on \(L\) (interior and up to the reflecting boundary), with
\(\lambda_{R} \leq \Lambda\), has constant \(C_{H}\) depending only on
\(L\), and a Harnack chain of bounded length joining a near-maximizer to
\(x_{0} \in L_{0}\) gives
\(\Phi\left( x_{0} \right) \geq c_{H}\sup_{L_{0}}\Phi \geq c_{H}/2\).
Inserting \(\psi = \Phi\) into the survival inequality above yields
\(P_{x_{0}}\left( \tau_{R} > t \right) \geq \left( c_{H}/\left( 2C_{0} \right) \right)e^{- \lambda_{R}t}\).
\end{proof}
\section{Sharp small-hole asymptotics}
The comparison estimate of Theorem 5.3 sharpens, under mild additional
regularity of the reflecting boundary, to a first-order asymptotic with
an explicit leading constant. The sharpening is not uniform in the naive
sense: the constant depends on how close the center sits to
\(\partial M\), and the boundary-layer structure is itself
dimension-dependent. In dimensions \(d \geq 3\) the relevant parameter
is the rescaled distance
\(\operatorname{dist}\left( x_{\varepsilon},\partial M \right)/\varepsilon\), and the
profile is a nontrivial function of it; in dimension \(d = 2\) the
logarithmic capacity regime is coarser, the transition is governed by
the logarithmic ratio
\(\log \operatorname{dist}\left( x_{\varepsilon},\partial M \right)/\log\varepsilon\), and
we treat it separately. Throughout this subsection we assume that
\(\partial M\) is uniformly of class \(C^{1,1}\); this suffices for the
first-order statements below, boundary curvature entering only the
\(o(1)\) remainder, and a second-order constant --- which would require
\(C^{2}\) --- is not pursued. The order-level condenser and
boundary-flattening estimates underlying Theorem 5.3 are collected in
Appendix A; the present subsection instead develops the sharp model
calculations that fix the first-order constants under this additional
regularity. We retain the capacity scale \(\kappa_{d}(\varepsilon)\) of
Section 2, write
\(\omega_{d - 1} = \left| \mathbb{S}^{d - 1} \right| = 2\pi^{d/2}/\Gamma(d/2)\)
for the area of the unit sphere, and set
\[\gamma_{d} = \left\{ \begin{matrix}
\pi, & d = 2, \\
\frac{1}{2}(d - 2)\,\omega_{d - 1}, & d \geq 3.
\end{matrix} \right.\ \]
\begin{lemma}[Euclidean spherical condenser]
Let
\(0 < \varepsilon < R\), and let \(u\) be the capacitary
potential of the spherical condenser
\(\left( \overline{B(0,\varepsilon)},\,\mathbb{R}^{d}\backslash B(0,R) \right)\),
that is, the harmonic function on
\(B(0,R)\backslash\overline{B(0,\varepsilon)}\) with \(u = 1\)
on \(\partial B(0,\varepsilon)\) and \(u = 0\) on
\(\partial B(0,R)\). Then

\[\int_{B(0,R)\backslash B(0,\varepsilon)}^{}|\nabla u|^{2}\, dx = \left\{ \begin{matrix}
\frac{2\pi}{\log(R/\varepsilon)}, & d = 2, \\
\frac{(d - 2)\,\omega_{d - 1}}{\varepsilon^{2 - d} - R^{2 - d}}, & d \geq 3.
\end{matrix} \right.\ \]

In particular, for \(d \geq 3\) the Newtonian capacity of a
ball is
\(cap\left( B(0,\varepsilon) \right) = (d - 2)\,\omega_{d - 1}\,\varepsilon^{d - 2}\),
and for every fixed \(R > 0\) one has, as
\(\varepsilon \rightarrow 0\),
\(\int|\nabla u|^{2}\, dx = (d - 2)\,\omega_{d - 1}\,\varepsilon^{d - 2}\left( 1 + O\left( (\varepsilon/R)^{d - 2} \right) \right)\)
for \(d \geq 3\) and
\(2\pi/\log(1/\varepsilon)\,\left( 1 + O\left( 1/\log(1/\varepsilon) \right) \right)\)
for \(d = 2\).
\end{lemma}
\begin{proof} By symmetry \(u\) is radial, and harmonicity gives
\(u(r) = A + Br^{2 - d}\) for \(d \geq 3\) and \(u(r) = A + B\log r\)
for \(d = 2\); the conditions \(u(\varepsilon) = 1\), \(u(R) = 0\) fix
\(A,B\). For \(d \geq 3\) this yields
\(u(r) = \left( r^{2 - d} - R^{2 - d} \right)/\left( \varepsilon^{2 - d} - R^{2 - d} \right)\),
hence
\(u'(r) = (2 - d)\, r^{1 - d}/\left( \varepsilon^{2 - d} - R^{2 - d} \right)\)
and
\[\int|\nabla u|^{2}\, dx = \omega_{d - 1}\int_{\varepsilon}^{R}u'(r)^{2}\, r^{d - 1}\, dr = \frac{(d - 2)^{2}\,\omega_{d - 1}}{\left( \varepsilon^{2 - d} - R^{2 - d} \right)^{2}}\int_{\varepsilon}^{R}r^{1 - d}\, dr = \frac{(d - 2)\,\omega_{d - 1}}{\varepsilon^{2 - d} - R^{2 - d}},\]
using
\(\int_{\varepsilon}^{R}r^{1 - d}\, dr = \left( \varepsilon^{2 - d} - R^{2 - d} \right)/(d - 2)\).
The case \(d = 2\) is identical with
\(u(r) = \log(R/r)/\log(R/\varepsilon)\), giving
\(2\pi/\log(R/\varepsilon)\). The stated expansions follow on fixing
\(R\) and letting \(\varepsilon \rightarrow 0\). \end{proof}
\begin{lemma}[Reflected condenser and the boundary-layer profile, \(d \geq 3\)]
Let \(d \geq 3\), let
\(H = \{ x \in \mathbb{R}^{d}:x_{d} > 0\}\) carry the Neumann
condition on \(\partial H = \{ x_{d} = 0\}\), and for
\(a \in \lbrack 0,\infty)\) and \(0 < \varepsilon < R\)
let
\(T_{a}(\varepsilon) = B\left( a\varepsilon e_{d},\varepsilon \right) \cap H\)
with reflecting condenser energy
\(E_{H}\left( T_{a}(\varepsilon);R \right)\) in
\(B(0,R) \cap H\). With \(E_{0}(\varepsilon;R)\) the
Euclidean energy of \(B(0,\varepsilon)\) from Lemma D.1, the
limit
\(\beta(a): = \lim_{\varepsilon/R \rightarrow 0}E_{H}\left( T_{a}(\varepsilon);R \right)/E_{0}(\varepsilon;R)\)
exists and defines a continuous function
\(\beta:\lbrack 0,\infty\rbrack \rightarrow \left\lbrack \frac{1}{2},1 \right\rbrack\)
with \(\beta(0) = \frac{1}{2}\) and
\(\beta(\infty) = 1\).
\end{lemma}
\begin{proof} A function on \(H\) with vanishing Neumann trace extends
by even reflection \(x_{d} \mapsto - x_{d}\) to \(\mathbb{R}^{d}\) with
equal Dirichlet energy on each half-space, so the reflecting problem for
\(T_{a}(\varepsilon)\) is the symmetric Euclidean condenser for
\(T_{a}(\varepsilon) \cup T_{a}(\varepsilon)^{*}\) and
\(E_{H}\left( T_{a}(\varepsilon);R \right) = \frac{1}{2}E_{\mathbb{R}^{d}}\left( T_{a}(\varepsilon) \cup T_{a}(\varepsilon)^{*};R \right)\).
Scaling \(x \mapsto \varepsilon x\) shows the ratio depends only on
\(a\) in the limit, and continuity follows from continuity of condenser
energy under the Hausdorff variation \(a \mapsto T_{a}(1)\). At
\(a = 0\) the reflected set is the full ball \(B(0,\varepsilon)\),
giving \(\beta(0) = \frac{1}{2}\); as \(a \rightarrow \infty\) the
reflected copies separate, the mutual Newtonian energy vanishing with
separation, so the union capacity tends to twice a single ball's and
\(\beta(\infty) = 1\). The reflected set lies between these
configurations, so \(\frac{1}{2} \leq \beta(a) \leq 1\). For \(d = 2\)
the logarithmic capacity is insensitive to a finite rescaled separation,
so this \(a\)-profile degenerates; the planar case is treated in Lemma
D.7. \end{proof}
\begin{lemma}[Local model convergence]
Assume
\(\partial M\) uniformly of class \(C^{1,1}\), and let
\(x_{\varepsilon} \in M\) satisfy
\(\operatorname{dist}\left( x_{\varepsilon},\partial M \right)/\varepsilon \rightarrow a \in \lbrack 0,\infty\rbrack\).
For fixed \(R > 1\) let \(D_{\varepsilon}(R)\) be the
Dirichlet energy of the capacitary potential of
\(B_{M}\left( x_{\varepsilon},\varepsilon \right)\) in
\(B_{M}\left( x_{\varepsilon},R\varepsilon \right)\backslash\overline{B_{M}\left( x_{\varepsilon},\varepsilon \right)}\)
with the Neumann condition on \(\partial M\), and
\(E_{a}(\varepsilon;R)\) the energy of the corresponding model
condenser at radii \((\varepsilon,R\varepsilon)\) --- the free
Euclidean condenser when \(a = \infty\), the reflected half-space
condenser at height \(a\) when \(a < \infty\). Then
\(D_{\varepsilon}(R) = E_{a}(\varepsilon;R)\left( 1 + o(1) \right)\)
as \(\varepsilon \rightarrow 0\) with \(R\) fixed,
uniformly along the sequence.
\end{lemma}
\begin{proof} Under uniform \(C^{1,1}\) regularity the second
fundamental form of \(\partial M\) is uniformly bounded and the reach of
\(M\) is uniformly positive, so geodesic and Euclidean
\(\varepsilon\)-spheres lie within Hausdorff distance
\(O\left( \varepsilon^{2} \right)\); hence
\(B_{M}\left( x_{\varepsilon},\varepsilon \right)\) and
\(B\left( x_{\varepsilon},\varepsilon \right) \cap M\) differ by a shell
of relative capacity \(o(1)\), and we use the Euclidean ball. Choose
\(C^{1,1}\) coordinates \(\Phi\) on
\(B_{M}\left( x_{\varepsilon},\rho_{0} \right)\),
\(\rho_{0} < \operatorname{reach}(M)\), with
\(\Phi\left( x_{\varepsilon} \right) = 0\),
\(D\Phi\left( x_{\varepsilon} \right) = \operatorname{Id}\),
\(\Phi(\partial M) \subset \{ y_{d} = 0\}\), and
\(\parallel D\Phi(y) - \operatorname{Id} \parallel \leq L\left| y - x_{\varepsilon} \right|\),
\(L\) uniform; rescale
\(z = \varepsilon^{- 1}\Phi\left( x_{\varepsilon} + \varepsilon\, \cdot \, \right)\),
under which the pulled-back metric obeys
\(\parallel g_{\varepsilon}(z) - \operatorname{Id} \parallel \leq 2LR\varepsilon\) on
\(\{|z| \leq R\}\). We prove the two matching variational bounds.
\emph{Upper bound.} Let \(v\) be the model capacitary potential on
\(\{|z| \leq R\}\). Its pullback
\(v \circ \left( \varepsilon^{- 1}\Phi \right)\) is admissible for the
curved condenser, and the metric estimate gives
\(D_{\varepsilon}(R) \leq \int|\nabla v|_{g_{\varepsilon}}^{2}\, d{\operatorname{vol}}_{g_{\varepsilon}} \leq E_{a}(\varepsilon;R)\,(1 + CLR\varepsilon)\).
\emph{Lower bound.} Let \(w_{\varepsilon}\) be the curved capacitary
potential and
\({\widetilde{w}}_{\varepsilon} = w_{\varepsilon} \circ \left( \varepsilon^{- 1}\Phi \right)^{- 1}\)
its flattened rescaling. The energies
\(\{\int\left| \nabla{\widetilde{w}}_{\varepsilon} \right|^{2}\}\) are
bounded, so along any subsequence
\({\widetilde{w}}_{\varepsilon} \rightharpoonup w_{*}\) weakly in
\(H^{1}\left( \{|z| \leq R\} \right)\); the Hausdorff convergence of the
rescaled target and of the flattened Neumann boundary makes \(w_{*}\)
admissible for the model condenser, and lower semicontinuity together
with the metric estimate gives
\(\liminf D_{\varepsilon}(R)\,(1 + CLR\varepsilon)^{- 1} \geq \int\left| \nabla w_{*} \right|^{2} \geq E_{a}(\varepsilon;R)\).
(Equivalently, the Dirichlet forms Mosco-converge under the smooth
flattening.)
Combining,
\(D_{\varepsilon}(R) = E_{a}(\varepsilon;R)\left( 1 + O(LR\varepsilon) \right) = E_{a}(\varepsilon;R)\left( 1 + o(1) \right)\)
for \(R\) fixed, the constants uniform by uniform admissibility.
\end{proof}
\begin{proposition}[Sharp Neumann capacity of small intrinsic
balls, \(d \geq 3\)]
Let \(d \geq 3\) and
\(x_{\varepsilon} \in M\) with
\(\operatorname{dist}\left( x_{\varepsilon},\partial M \right)/\varepsilon \rightarrow a \in \lbrack 0,\infty\rbrack\).
Then
\(\operatorname{Cap}_{N}\left( B_{M}\left( x_{\varepsilon},\varepsilon \right) \right) = \beta(a)\,\gamma_{d}\,|M|^{- 1}\kappa_{d}(\varepsilon)\left( 1 + o(1) \right)\).
\end{proposition}
\begin{proof} Fix \(R > 1\) and split the energy of the equilibrium
potential \(w\) into the local part on
\(B_{M}\left( x_{\varepsilon},R\varepsilon \right)\) and the remainder.
By Lemma D.3 the local part is
\(E_{a}(\varepsilon;R)\left( 1 + o(1) \right)\), and by Lemmas D.1--D.2,
\(E_{a}(\varepsilon;R) = \beta(a)\, cap\left( B_{\varepsilon} \right)\left( 1 + o_{R}(1) \right) = 2\beta(a)\gamma_{d}\kappa_{d}(\varepsilon)\left( 1 + o_{R}(1) \right)\)
as \(R \rightarrow \infty\). The remainder: \(w\) is harmonic outside
with boundary value \(O\left( R^{2 - d} \right)\), so its energy
together with the mean-zero shift is
\(o\left( \kappa_{d}(\varepsilon) \right)\) uniformly for
\(R \geq R_{0}\). Letting \(\varepsilon \rightarrow 0\) then
\(R \rightarrow \infty\) and dividing by \(2|M|\) (the normalization
\(\mathcal{E}(w,w) = \left( 2|M| \right)^{- 1}\int_{M}^{}|\nabla w|^{2}\))
gives the claim, uniformly along the sequence. \end{proof}
\begin{proposition}[First-order capacity--eigenvalue relation]
The first-order relation stated in Proposition 5.1 is established there
using only the Neumann spectral gap, with no boundary regularity beyond
Assumption 2.5. We record it separately here because it is the precise
input to the sharp asymptotics of this subsection and, through Lemma
7.5, to the lower bounds of Section 7.
\end{proposition}
\begin{theorem}[Sharp small-hole eigenvalue asymptotic, \(d \geq 3\)]
Let \(d \geq 3\), \(\partial M\)
uniformly \(C^{1,1}\), and \(x_{\varepsilon} \in M\)
with
\(\operatorname{dist}\left( x_{\varepsilon},\partial M \right)/\varepsilon \rightarrow a \in \lbrack 0,\infty\rbrack\).
Then

\[\lambda\left( B_{M}\left( x_{\varepsilon},\varepsilon \right) \right) = \frac{\beta(a)\,\gamma_{d}}{|M|}\,\kappa_{d}(\varepsilon)\,\left( 1 + o(1) \right) = \beta(a)\,\frac{(d - 2)\,\omega_{d - 1}}{2|M|}\,\varepsilon^{d - 2}\left( 1 + o(1) \right).\]

Interior sequences \((a = \infty)\) give \(\beta = 1\)
and boundary sequences
\(\left( a = 0,\ x_{\varepsilon} \in \partial M \right)\) give
\(\beta = \frac{1}{2}\), a factor of two between the two.
\end{theorem}
\begin{proof} Immediate from Propositions D.4 and D.5. \end{proof}
\begin{lemma}[Planar two-disk logarithmic capacity and
boundary-layer profile, \(d = 2\)]
For
\(0 < \varepsilon \ll \delta\) the logarithmic capacity of the
union of two disks of radius \(\varepsilon\) whose centers are at
distance \(2\delta\) is
\(\operatorname{cap}_{\log} = \sqrt{2\varepsilon\delta}\,\left( 1 + o(1) \right)\).
Consequently, writing
\(\delta_{\varepsilon} = \operatorname{dist}\left( x_{\varepsilon},\partial M \right)\)
and assuming
\(\log\delta_{\varepsilon}/\log\varepsilon \rightarrow \theta \in \lbrack 0,1\rbrack\),
the reflected half-plane condenser energy of
\(B\left( x_{\varepsilon},\varepsilon \right) \cap M\) is
\(E_{H} = B(\theta)\, E_{0}\,\left( 1 + o(1) \right)\) with the
planar profile

\[B(\theta) = \frac{1}{1 + \theta},\quad\quad B(1) = \frac{1}{2},\quad B(0) = 1.\]
\end{lemma}
\begin{proof} By symmetry the equilibrium measure of the two-disk set
carries mass \(\frac{1}{2}\) on each disk; the potential at a center is
\(\frac{1}{2}\log(1/\varepsilon)\) from its own ring and
\(\frac{1}{2}\log\left( 1/(2\delta) \right)\left( 1 + o(1) \right)\)
from the distant disk, so the Robin constant is
\(\frac{1}{2}\log\left( 1/(2\varepsilon\delta) \right)\left( 1 + o(1) \right)\)
and
\(\operatorname{cap}_{\log} = \sqrt{2\varepsilon\delta}\,\left( 1 + o(1) \right)\).
This is the standard logarithmic-capacity (transfinite-diameter)
asymptotic for two small conductors of radius \(\varepsilon\) at
separation \(2\delta\) (see, e.g., {[}25{]}). By even reflection (Lemma
D.2) the half-plane energy is half the full-plane condenser energy of
the reflected cluster, whose effective logarithmic radius is
\(\sqrt{2\varepsilon\delta_{\varepsilon}} = \varepsilon^{(1 + \theta)/2 + o(1)}\);
with \(E_{0} = 2\pi/\log(1/\varepsilon)\left( 1 + o(1) \right)\) this
gives
\(E_{H} = \frac{1}{2} \cdot 2\pi/\log\left( 1/\sqrt{2\varepsilon\delta_{\varepsilon}} \right)\left( 1 + o(1) \right) = 2\pi/\left( (1 + \theta)\log(1/\varepsilon) \right)\left( 1 + o(1) \right) = B(\theta)E_{0}\left( 1 + o(1) \right)\).
For \(\delta_{\varepsilon} = O(\varepsilon)\) the cluster has scale
\(\varepsilon\) and the leading coefficient is that of a single disk,
giving \(B = \frac{1}{2}\) \((\theta = 1)\); for
\(\delta_{\varepsilon}\) bounded below the boundary is macroscopically
distant and \(B = 1\) \((\theta = 0)\). \end{proof}
\begin{theorem}[Sharp small-hole eigenvalue asymptotic, \(d = 2\)]
Let \(d = 2\) and \(\partial M\)
uniformly of class \(C^{1,1}\), and write
\(\delta_{\varepsilon} = \operatorname{dist}\left( x_{\varepsilon},\partial M \right)\).
The planar boundary layer has three regimes:

(i) boundary-scale centers,
\(\delta_{\varepsilon}/\varepsilon = O(1)\) (including
\(\delta_{\varepsilon} = 0\)):

\[\lambda\left( B_{M}\left( x_{\varepsilon},\varepsilon \right) \right) = \frac{1}{2} \cdot \frac{\pi}{|M|\,\log(1/\varepsilon)}\,\left( 1 + o(1) \right);\]

(ii) intermediate centers,
\(\varepsilon \ll \delta_{\varepsilon} \ll 1\) with
\(\theta: = limlog\left( 1/\delta_{\varepsilon} \right)/\log(1/\varepsilon) \in (0,1)\)
(equivalently
\(\delta_{\varepsilon} = \varepsilon^{\theta + o(1)}\)):

\[\lambda\left( B_{M}\left( x_{\varepsilon},\varepsilon \right) \right) = B(\theta) \cdot \frac{\pi}{|M|\,\log(1/\varepsilon)}\,\left( 1 + o(1) \right),\quad\quad B(\theta) = \frac{1}{1 + \theta};\]

(iii) macroscopically interior centers, \(\delta_{\varepsilon}\)
bounded below:

\[\lambda\left( B_{M}\left( x_{\varepsilon},\varepsilon \right) \right) = \frac{\pi}{|M|\,\log(1/\varepsilon)}\,\left( 1 + o(1) \right).\]

Here \(\theta \in \lbrack 0,1\rbrack\) is the logarithmic
boundary-layer parameter, with \(\theta = 0\) for macroscopically
interior centers (iii) and \(\theta = 1\) for boundary-scale
centers (i); the three regimes are the single asymptotic
\(\lambda\left( B_{M}\left( x_{\varepsilon},\varepsilon \right) \right) = B(\theta)\,\pi/\left( |M|\log(1/\varepsilon) \right)\,\left( 1 + o(1) \right)\)
with \(B(\theta) = 1/(1 + \theta)\). The profile \(B\)
acts on the eigenvalue --- and, by Proposition D.5, on the Neumann
capacity --- decreasing from the interior value \(B(0) = 1\) to
the boundary value \(B(1) = \frac{1}{2}\); the associated
hitting-time constant of Corollary D.9 is the reciprocal
\(1/B(\theta) = 1 + \theta\), increasing from \(1\) to
\(2\). Thus a boundary-scale target has half the eigenvalue, and
twice the hitting time, of a macroscopically interior one. In contrast
to \(d \geq 3\), every fixed rescaled distance
\(\delta_{\varepsilon}/\varepsilon\) corresponds to
\(\theta = 1\) and hence to the same boundary value; the
nontrivial planar transition occurs only on the logarithmic scale
\(\theta\).
\end{theorem}
\begin{proof} By even reflection across the locally flattened boundary,
\(\lambda\left( B_{M}\left( x_{\varepsilon},\varepsilon \right) \right) = \lambda^{\widetilde{M}}(\, \cdot \,)\)
on the double \(\widetilde{M}\),
\(\left| \widetilde{M} \right| = 2|M|\), the target becoming the
reflected two-disk cluster of Lemma D.7. By Proposition D.5 and the
planar capacity normalization
\(\operatorname{Cap}_{N}(A) = \pi/\left( |M|\log\left( 1/\operatorname{cap}_{\log}A \right) \right)\left( 1 + o(1) \right)\),
the eigenvalue equals
\(\pi/\left( \left| \widetilde{M} \right|\log\left( 1/\sqrt{2\varepsilon\delta_{\varepsilon}} \right) \right)\left( 1 + o(1) \right) = B(\theta)\,\pi/\left( |M|\log(1/\varepsilon) \right)\left( 1 + o(1) \right)\),
which is regime (ii). In regime (i) Lemma D.7 gives effective
logarithmic radius of order \(\varepsilon\), hence \(B = \frac{1}{2}\)
(\(\theta = 1\)); in regime (iii) no reflection is needed and the
standard planar small-hole asymptotic gives \(B = 1\) (\(\theta = 0\)).
\end{proof}
\begin{corollary}[Recovery constants: sharp profile for fixed
families, bracketed for the universal cover]
Fix targets
\(B_{M}\left( x_{i},\varepsilon \right)\),
\(i = 1,\ldots,m\), and let \(\Lambda_{i}\) denote the
sharp small-hole constant of Theorem D.6 (for \(d \geq 3\),
\(\Lambda_{i} = \beta\left( a_{i} \right)\) with
\(a_{i} = \lim\,\operatorname{dist}\left( x_{i},\partial M \right)/\varepsilon\)) or
Theorem D.8 (for \(d = 2\),
\(\Lambda_{i} = B\left( \theta_{i} \right)\) with
\(\theta_{i} = \lim\,\log\,\operatorname{dist}\left( x_{i},\partial M \right)/\log\varepsilon\)),
so that the killed eigenvalues satisfy
\(\lambda_{i}(\varepsilon) \sim \gamma_{d}|M|^{- 1}\Lambda_{i}\,\kappa_{d}(\varepsilon)\).
The genuinely sharp finite-family information is the missed-target
profile of Corollary 4.5,

\[\Phi_{A}(u) \sim \sum_{i = 1}^{m}\exp\left( - \frac{\gamma_{d}}{|M|}\,\Lambda_{i}\,\kappa_{d}(\varepsilon)\, u \right);\]

in particular, writing
\(\Lambda_{\min} = \min_{i}\Lambda_{i}\), the following
worst-target condition gives a clean high-probability threshold: for
failure probability \(\delta \in (0,1)\), hitting every target
with probability at least \(1 - \delta\) holds once

\[T \geq s_{0} + \frac{|M|}{\gamma_{d}\,\Lambda_{\min}}\,\kappa_{d}(\varepsilon)^{- 1}\left( \log m + \log\frac{C}{\delta} \right),\]
\end{corollary}
\(C\) \emph{absorbing the heat-kernel domination constant of Section 4.
For heterogeneous families the exact threshold is governed by the full
profile} \(\Phi_{A}\)\emph{, not by} \(\Lambda_{\min}\) \emph{alone. For
the universal full cover} \(\left( m = N_{\varepsilon}(M) \right)\)
\emph{the leading constant} \(C_{cover}\) \emph{in}
\(T \sim s_{0} + C_{cover}\,\kappa_{d}(\varepsilon)^{- 1}\log N_{\varepsilon}(M)\)
\emph{is bracketed by the interior and boundary constants,}
\[\frac{|M|}{\gamma_{d}} \leq C_{cover} \leq \frac{2|M|}{\gamma_{d}},\]
\emph{its exact value depending on a boundary-layer covering
optimization for asymptotically optimal intrinsic}
\(\varepsilon\)\emph{-nets together with a matching minimax lower bound,
and is not determined here.}
\begin{proof} Insert the sharp eigenvalues of Theorems D.6 and D.8 into
the union bound of Corollary 4.5 for \(\Phi_{A}\); the term
\(\log(C/\delta)\) is the fixed-confidence contribution, present already
for \(m = 1\). For the universal cover the threshold is governed by
\(\min_{\text{centers}}\Lambda \in \left\lbrack \frac{1}{2},1 \right\rbrack\),
giving
\(C_{cover} = |M|/\left( \gamma_{d}\min_{\text{centers}}\Lambda \right) \in \left\lbrack |M|/\gamma_{d},2|M|/\gamma_{d} \right\rbrack\).
\end{proof}
Theorems D.6 and D.8 are genuine two-sided asymptotics for the
eigenvalue, but the matching leading constant in the universal recovery
threshold is deliberately not asserted: pinning \(C_{cover}\) would
require both the boundary-layer covering optimization above and a
minimax lower bound carrying the same constant, a sharpening of the
order-level bounds of Section 7 that we do not undertake. We emphasize
that the sharp-constant results of this subsection rest on the
additional uniform \(C^{1,1}\) regularity assumed at its outset, and are
not used in the proofs of Theorems 6.1--6.3, which rely only on the
comparison estimate of Theorem 5.3, for which boundary uniformity
already suffices. The one relation connected to this subsection that
does enter the main development is the first-order relation recalled in
Proposition D.5: proved in Proposition 5.1 using only the spectral gap,
it supplies the lower bounds of Section 7 through Lemma 7.5. The role of
this subsection is to identify, at the sharp analytic level, the
explicit constant \(\gamma_{d}/|M|\), the boundary-layer profile ---
Newtonian in \(d \geq 3\) and logarithmic in \(d = 2\) --- and the
factor-of-two spread between interior and boundary targets that any
future sharp-constant recovery theory must incorporate.
\begin{figure}[htbp]
\centering
\robustfig[0.98\linewidth]{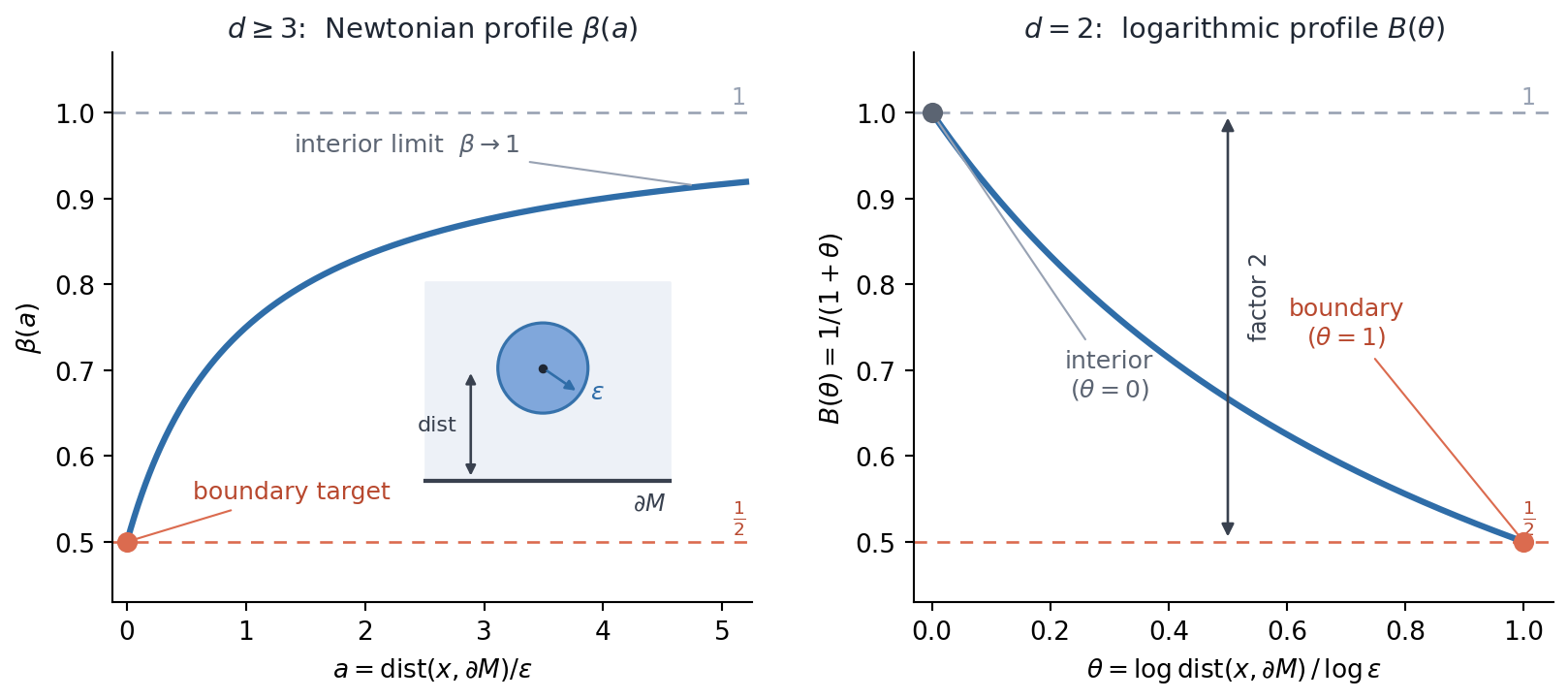}
\caption{Boundary-layer profiles for the small-hole eigenvalue. Under the additional \(C^{1,1}\) regularity of this subsection, the sharp leading constant of the killed eigenvalue \(\lambda\left( B_{M}(x,\varepsilon) \right)\) carries a boundary-layer profile set by how close the target sits to the reflecting boundary \(\partial M\). \textbf{Left (}\(d \geq 3\)\textbf{):} the Newtonian profile \(\beta(a)\), a function of the rescaled distance \(a = \operatorname{dist}(x,\partial M)/\varepsilon\) (inset), rises from the boundary value \(\beta(0) = \frac{1}{2}\) to the interior limit \(\beta(\infty) = 1\); the endpoints and the tail \(\beta(a) = 1 - \frac{1}{2a} + \cdots\) are exact, while the connecting curve is schematic. \textbf{Right (}\(d = 2\)\textbf{):} the logarithmic profile \(B(\theta) = 1/(1 + \theta)\), a function of \(\theta = \log \operatorname{dist}(x,\partial M)/\log\varepsilon\), decreases from the interior value \(B(0) = 1\) to the boundary value \(B(1) = \frac{1}{2}\). In both dimensions the profile spans a factor of two, so a boundary-scale target has half the eigenvalue --- and twice the hitting time --- of a macroscopically interior one (Theorems D.6 and D.8).}
\label{fig:4}
\end{figure}

\end{document}